\makeatletter
\def\thanks#1{\protected@xdef\@thanks{\@thanks
        \protect\footnotetext{#1}}}
\makeatother

\documentclass{article}

\usepackage{authblk}
\usepackage{amsmath}
\usepackage{amsfonts,amssymb}
\usepackage{mathrsfs}
\usepackage{ntheorem}
\usepackage{indentfirst}
\usepackage{cite}
\usepackage{geometry}
\usepackage{graphicx}
\usepackage{subfigure}
\usepackage{epstopdf}
\usepackage{authblk}
\usepackage{float}
\usepackage{color}
\newtheorem{theorem}{Theorem}[section]
\newtheorem{lemma}{Lemma}[section]

\theoremstyle{definition}
\newtheorem{definition}{Definition}[section]

\theoremstyle{remark}
\newtheorem{remark}{Remark}[section]

\newenvironment{proof}[1][Proof]{\noindent\textbf{#1.} }

\begin{document}
\title{\textbf{Incomplete Information Linear-Quadratic Mean-Field Games and Related Riccati Equations}
\thanks{This work was supported by the National Key R\&D Program of China (No. 2022YFA1006104), the National Natural Science Foundation of China (Nos. 12022108, 11971267, 12001320, 11831010, 61961160732, 61977043), the China Postdoctoral Science Foundation (No. 2023M732085), the Natural Science Foundation of Shandong Province (Nos. ZR2022JQ01, ZR2019ZD42, ZR2020ZD24), the Taishan Scholars Young Program of Shandong (No. TSQN202211032), the Distinguished Young Scholars Program and the Young Scholars Program of Shandong University. (\emph{Corresponding author: Ke Yan}.)}
}

\author{Min Li,~ Tianyang Nie, ~Shunjun Wang ~and~ Ke Yan \thanks{Min Li is with the School of Mathematics, Shandong University, Jinan 250100, China, and also with the Geotechnical and Structural Engineering Research Center, Shandong University, Jinan 250061, China (e-mail: lim@sdu.edu.cn).} \thanks{Tianyang Nie and Ke Yan are with the School of Mathematics, Shandong University, Jinan 250100, China (e-mail: nietianyang@sdu.edu.cn; 201812092@mail.sdu.edu.cn).}
\thanks{Shujun Wang is with the School of Management, Shandong University, Jinan 250100, China (e-mail: wangshujun@sdu.edu.cn).}
}

\date{}

\maketitle

\begin{abstract}
We study a class of linear-quadratic mean-field games  with incomplete information. For each agent, the state is given by a linear forward stochastic differential equation with common noise. Moreover, both the state and control variables can enter the diffusion coefficients of the state equation.
We deduce the open-loop adapted decentralized strategies and feedback decentralized strategies by mean-field forward-backward stochastic differential equation  and Riccati equations, respectively. The well-posedness of the corresponding consistency condition system is obtained and the limiting state-average turns out to be the solution of a mean-field stochastic differential equation driven by common noise.
We also verify the $\varepsilon$-Nash equilibrium property of the decentralized control strategies. Finally, a network security problem is studied to illustrate our results as an application.
 \end{abstract}


\section{Introduction}
\label{sec:introduction}
The present work of mean-field games (MFGs) were introduced by Lasry and Lions \cite{LL2007} and simultaneously by Huang, Malham\'e and Caines \cite{HMC2006}.
The asymptotic Nash equilibrium for stochastic differential games with infinite number of agents subject to a mean-field interaction when the number of agents goes to infinity leads to the theory of MFGs.
In recent years, it has widely application in reality, such as finance, economics and engineering, etc.
 The interested readers can refer to can refer to \cite{BP2014,HHL2018,HHN2018,MB2017,WZ2017} for linear-quadratic (LQ) MFGs, and refer to \cite{BFY2013,BDLP2009,C2010,CD2013} for further analysis of MFGs and related topics.

In the real world, usually agents can only get incomplete information at most cases.
Wang, Wu, and Xiong \cite{WWX2018} studied the optimal control problems for forward-backward stochastic differential equation (FBSDE) with incomplete information and introduced a backward separation idea to overcome the difficulty arising from an LQ optimal control problem.
Recently, the large population problems with incomplete information are extensively studied. For example, Huang, Caines and Malham\'e \cite{HCM2006} investigated dynamic games in a large population of stochastic agents which have local noisy measurements of its own state.
Huang, Wang and Wu \cite{HWW2016} studied the backward mean-field linear-quadratic-Gaussian (LQG) games of weakly coupled stochastic large population system with full and partial information. \c{S}en and Caines \cite{SC2019} considered the nonlinear MFGs where an individual agent has noisy observation on its own state. Furthermore, some literature can be found in \cite{BCL2017,BLM2017,LF2019,WXX2017,WZZ2013}
 for the mean-filed type control problems with incomplete information.

The above mentioned literatures about MFGs for stochastic large population systems are independent of the random factors in each agent's state process. However, many models in applications do not satisfy this assumption. For example, the financial market models often consider some common market noise affecting the agents.  In \cite{CFS2015}, Carmona, Fouque and Sun  gave an explicit example of a LQ mean-field model with common noise used to interbank lending and borrowing. The presence of common noise which means the fact that they are influenced by common information e.g. public data. Consequently, the model of MFGs with common noise is
a general setting in reality and have attracted significant attentions recently.

In our paper, we consider $N$-player game models in which individual agents are subjected to two independent sources of noise: an idiosyncratic noise, independent from one individual to another, and a homogenous one, identical to all the players, accounting for the common environment in which the individual states evolve. This is mainly due to the fact that the states of the individual agents are subjected to correlated noise term.

Let us recall the literatures on MFGs with common noise. Carmona, Delarue and Lacker \cite{CDL2016} discussed the existence and uniqueness of an equilibrium in the presence of a common noise, see also \cite{CD2018}, Volume \uppercase\expandafter{\romannumeral2}, Chapter 3.
Using PDE approach, Cardaliaguet, Delarue, Lasry and Lions \cite{CDLL2019} studied the master equation and convergence of MFGs with common noise. Tchuendom \cite{T2018} showed that a common noise may restore uniqueness in LQ MFGs. Bayraktar, Cecchin, Cohen and Delaure \cite{BCCD2022} focused on finite state MFGs with common noise. Mou and Zhang \cite{MZ2020} considered non-smooth data MFGs and global well-posedness of master equations with common noise. Li, Mou, Wu, Zhou \cite{LMWZ2022} discussed a class of LQ mean field games of  controls with common noise, which proved the global well-posedness of the corresponding master equation without any monotonicity conditions.
The literatures related to our work mainly include \cite{BFH2021,HY2021,HW2016}. Huang and Wang \cite{HW2016} studied the dynamic optimization of large population system with partial information and common noise. Bensoussan, Feng and Huang \cite{BFH2021} considered a class of LQG MFGs with partial observation structure for individual agents and the dynamic of individuals are driven by common noise.
Huang and Yang \cite{HY2021} investigated asymptotic solvability of a LQ mean-field social optimization problem with controlled diffusions and indefinite state and control weights, where the state is driven by the individual noise and common noise.

The contributions of our paper are the following.
\begin{itemize}
\item Firstly, we assume the states of all agents are governed by some underlying common noise, so the individual agents are not independent of each other. Compared with our recent paper \cite{LNW2022}, where the state of all agents depend on two independent noises, it leads to the state-average limit is deterministic. In current paper, with the help of conditional expectation, the presence of common noise makes the state-average limit in MFGs analysis be some stochastic process instead of deterministic process, i.e., the dynamic of the limit of state-average satisfies a stochastic differential equation (SDE) driven by the common noise, see Remark \ref{remark for m} and \eqref{m-1} for more details.
\item Secondly, we study LQ MFGs with common noise, in which the state and the control variable can both enter the diffusion coefficients in front of  individual noise and common noise for individual state. Compared with \cite{HW2016}, the diffusion terms in  \cite{HW2016} are independent of the state and control variables. Our model is more complicated than the one in \cite{HW2016}.  Due to the dependence of control variables of  diffusion coefficients, our Riccati equations are no longer in standard form. Thus, there arise difficulties when we represent the optimal feedback control strategies of open-loop adapted policies via some Riccati equations.
\item Thirdly, we introduce two Riccati equations to decompose the consistency condition system of our LQ MFGs. One Riccati equation is introduced due to the mean-field interaction, see \eqref{Gamma-1}. Another Riccati equation, see \eqref{P-1}, is no longer in standard form as in Yong and Zhou \cite{YZ1999}, since both the diffusion coefficients in front of individual noise and common noise depend on the control variable.  We are successful to establish the well-posedness of these two Riccati equations under suitable assumptions, see Lemma \ref{Plemma}, Theorem \ref{RCCtheorem} and  Theorem \ref{RCCtheorem Gamma}.
\item Finally, we apply the results to solving a network security problem. An explicit form of the unique approximate Nash equilibrium is obtained through Riccati equations.
\end{itemize}

The rest of the paper is organized as follows. In Section \ref{sec2}, we formulate the LQ MFG problem with incomplete information.
%The control strategies are obtained by Pontryagin's stochastic maximum principle and consistency conditions are also established though MF-FBSDE.
Moreover, we introduce the Riccati equations to decompose the consistency condition system and verify the $\varepsilon$-Nash equilibrium of the control strategies.
Section \ref{sec3} establishes the well-posedness of this new kind of Riccati equation and the strategies are given explicitly by solutions of Riccati equations.
In Section \ref{sec4}, we give an example for the network security model.

\section{Incomplete Information LQ MFGs}\label{sec2}
For fixed $T >0$, we consider a finite time interval $[0,T]$. Let $(\Omega,\mathcal{F},\mathbb{F},\mathbb{P})$ be a complete filtered probability space satisfying the usual conditions on which a standard $(1+N)$-dimensional Brownian motion $\{W_0(t),W_i(t), 1\leq i \leq N\}_{0\leq t\leq T}$ is defined.
Here, $W_i$ is the individual noise while $W_0$ is the common noise due to underlying common factors. Let $\mathbb{F}=\{\mathcal{F}_t\}_{0\leq t\leq T}$ be the natural filtration with $\mathcal{F}_t:=\sigma\{W_i(t),0\leq i \leq N, 0\leq t \leq T\} \vee \mathcal{N}_{\mathbb{P}}$ (where $\mathcal{N}_{\mathbb{P}}$ is the class of $\mathbb{P}$-null sets of $\mathcal{F}_t$).
Define $\mathcal{F}_t^{W_0}:=\sigma\{W_0(t),0\leq t \leq T\}\vee \mathcal{N}_{\mathbb{P}}$,
$\mathcal{F}_t^{W_i}:=\sigma\{W_i(t),0\leq t \leq T\}\vee \mathcal{N}_{\mathbb{P}}$,
$\mathcal{F}_t^{i}:=\sigma\{W_0(t),W_i(t),0\leq t \leq T\}\vee \mathcal{N}_{\mathbb{P}}$ and $\mathcal{G}_t:=\sigma\{W_i(t),1\leq i \leq N, 0\leq t \leq T\}\vee \mathcal{N}_{\mathbb{P}}$.
In our setting, we denote by
$\{\mathcal{F}_t^{W_0}\}_{0\leq t\leq T}$ the common information taking effects on all agents;
$\{\mathcal{F}_t^{W_i}\}_{0\leq t\leq T}$ the individual information of the $i$-th agent;
$\{\mathcal{F}_t^{i}\}_{0\leq t\leq T}$ the full information of the $i$-th agent;
$\{\mathcal{F}_t\}_{0\leq t\leq T}$ denotes the complete information of the system.
In our paper, the $i$-th agent cannot access the information of other agents and can only observe its individual noise $W_i(\cdot)$.

 Throughout the paper, $\mathbb{R}^{n}$ denotes the $n$-dimensional Euclidean space with its norm and inner product denoted by $|\cdot|$ and $\langle \cdot, \cdot \rangle$, respectively. For a given vector or matrix $M$, let $M^{\top}$ stand for its transpose. We denote the set of symmetric $n\times n$ (resp. positive semi definite) matrices with real elements by $\mathcal{S}^n$ ($\mathcal{S}_{+}^n$). If $M\in \mathcal{S}^n$ is positive (semi) definite, we write $M>(\geq)0$.
 For positive constant $k$, if $M\in \mathcal{S}^n$ and $M > kI$, we denote $M\gg0$.
 For a given Hilbert space $\mathcal{H}$ and a filtration $\{\mathcal{F}_t\}_{0\leq t \leq T}$, let $L^2_\mathcal{F}(0,T;\mathcal{H})$ denote the space of all $\mathcal{F}_t$-progressively measurable processes $g(\cdot)$ satisfying $\mathbb{E}\int_{0}^{T}|g(t)|^2dt<+\infty$; $L^2(0,T;\mathcal{H})$ denotes the space of all deterministic functions $g(\cdot)$ satisfying $\int_{0}^{T}|g(t)|^2dt<+\infty$; $L^{\infty}(0,T;\mathcal{H})$ denotes the space of uniformly bounded functions;
 $C([0,T];\mathcal{H})$ denotes the space of continuous functions.

Now, we consider a large population system with $N$ individual agents $\{\mathcal{A}_i\}_{1\leq i \leq N}$. The state $x_i(\cdot)$ for $i$-th agent $\mathcal{A}_i$ is given by the following linear SDE
\begin{equation}\label{state}
\left\{
\begin{aligned}
dx_{i}(t)=&[A(t)x_{i}(t)+B(t)u_{i}(t)+\alpha(t)x^{(N)}(t)+b(t)]dt \\
&+[C(t)x_{i}(t)+D(t)u_{i}(t)+\beta(t)x^{(N)}(t)+\sigma (t)]dW_{i}(t) \\
&+[C_0(t)x_{i}(t)+D_0(t)u_{i}(t)+\beta_0(t)x^{(N)}(t)+\sigma_0(t)]dW_{0 }(t), \\
x_{i}(0)=&x,
\end{aligned}
\right.
\end{equation}
where $x\in \mathbb{R}^n$ and $x^{(N)}(\cdot):=\frac{1}{N}\sum_{i=1}^N x_{i}(\cdot)$ denotes the state-average of population.
The corresponding coefficients $A(\cdot),B(\cdot),\alpha(\cdot),b(\cdot),C(\cdot),D(\cdot),\beta(\cdot),\sigma(\cdot), C_0(\cdot),D_0(\cdot),\beta_0(\cdot),$ $\sigma_0(\cdot)$ are some deterministic matrix-valued functions with appropriate dimensions.

For $1\leq i\leq N$, the centralized admissible control set $\mathcal{U}_{ad}^c$ is defined as
\begin{equation*}
  \mathcal{U}_{ad}^c:=\{u_i(\cdot)|u_i(\cdot)\in L^2_{\mathcal{G}_t}(0,T;\mathbb{R}^k)\}.
\end{equation*}
Moreover, for $1\leq i\leq N$, the decentralized admissible control set $\mathcal{U}_i$ is defined as
\begin{equation*}
  \mathcal{U}_i:=\{u_i(\cdot)|u_i(\cdot)\in L^2_{\mathcal{F}_t^{W_i}}(0,T;\mathbb{R}^k)\}.
\end{equation*}

Let $u(\cdot)=(u_1(\cdot),\ldots,u_N(\cdot))$ be the set of control strategies of all agents and $u_{-i}(\cdot)=(u_1(\cdot),\ldots,u_{i-1}(\cdot),u_{i+1}(\cdot),\ldots,u_N(\cdot))$ be the set of control strategies except for $i$-th agent $\mathcal{A}_i$.
The cost functional of $\mathcal{A}_i$ is given by
\begin{equation}\label{cost}
\begin{aligned}
&\mathcal{J}_{i}(u_{i}(\cdot ),u_{-i}(\cdot ))\\
&=\frac{1}{2}\mathbb{E}\Big\{\int_{0}^{T}\big[\langle
Q(t)(x_{i}(t)-x^{(N)}(t)),x_{i}(t)-x^{(N)}(t)\rangle+\langle R(t)u_{i}(t),u_{i}(t)\rangle \big]dt+\langle Gx_{i}(T),x_{i}(T)\rangle \Big\}.
\end{aligned}
\end{equation}

Moreover, we introduce the following assumptions of coefficients.

\begin{enumerate}
  \item[(A1)] %The coefficients of state equation satisfy
$
A(\cdot),\alpha(\cdot),C(\cdot),\beta(\cdot),C_0(\cdot),\beta_0(\cdot) \in L^{\infty }(0,T;\mathbb{R}^{n\times n});\\
b(\cdot),\sigma(\cdot),\sigma_0(\cdot) \in L^{\infty}(0,T;\mathbb{R}^{n});\\
B(\cdot),D(\cdot),D_0(\cdot) \in L^{\infty }(0,T;\mathbb{R}^{n\times k});$
  \item[(A2)]%The coefficients of cost functional satisfy
 $ Q(\cdot)\in L^{\infty }(0,T;\mathcal{S}^{n}),Q(\cdot)\geq 0,
 R\in L^{\infty }(0,T;\mathcal{S}^{k}),\\
 R(\cdot)\gg0,
 G\in \mathcal{S}^{n}, G\geq 0.$
\end{enumerate}

We mention that under assumptions \textup{(A1)}, the system of SDE \eqref{state} admits a unique solution. Under  assumptions \textup{(A2)}, the cost functional \eqref{cost} is well-defined.

Now, we formulate the following incomplete information large population problem and aim to find Nash equilibrium.\\
\textbf{Problem (\uppercase\expandafter{\romannumeral1}).}
For $1 \leq i \leq N$, finding the control strategy set $\bar{u}(\cdot)=(\bar{u}_{1}(\cdot),\ldots ,\bar{u}_{N}(\cdot))$ such that
\begin{equation*}
\mathcal{J}_{i}(\bar{u}_{i}(\cdot ),\bar{u}_{-i}(\cdot ))=\underset{
u_{i}\left( \cdot \right) \in \mathcal{U}_{ad}^c}{\inf }\mathcal{J}
_{i}(u_{i}(\cdot ),\bar{u}_{-i}(\cdot )),
\end{equation*}
where $\bar{u}_{-i}(\cdot)=(\bar{u}_1(\cdot),\ldots,\bar{u}_{i-1}(\cdot),\bar{u}_{i+1}(\cdot),\ldots,\bar{u}_N(\cdot))$.
For $1\leq i\leq N$, we  call $\bar{u}_i(\cdot)$ the Nash equilibrium of \textbf{Problem (\uppercase\expandafter{\romannumeral1})}. Moreover, the corresponding state $\bar{x}_i(\cdot)$ is called the optimal centralized trajectory.

Due to the coupling state-average $x^{(N)}(\cdot)$, it is complicated to study \textbf{Problem (\uppercase\expandafter{\romannumeral1})}. We will use MFG theory to seek the approximate Nash equilibrium, which bridges the ``centralized'' LQ games to the limiting state-average, as the number of agents $N$ tends to infinity. It is usually to replace the state-average by its frozen limit term.

As $N \rightarrow +\infty$, suppose $\bar{x}^{(N)}(\cdot)=\frac{1}{N}\sum_{i=1}^{N}\bar{x}_i(\cdot)$ is approximated by some processes $m(\cdot)$ which will be determined later by the consistency condition system.
We introduce the following auxiliary state for $\mathcal{A}_i$
\begin{equation}\label{z-lstate}
\left\{
\begin{aligned}
dz_{i}(t)=&[A(t)z_{i}(t)+B(t)u_{i}(t)+\alpha(t)m(t)+b(t)]dt \\
&+[C(t)z_{i}(t)+D(t)u_{i}(t)+\beta(t)m(t)+\sigma(t)]dW_{i}(t) \\
&+[C_0(t)z_{i}(t)+D_0(t)u_{i}(t)+\beta_0(t)m(t)+\sigma_0(t)]dW_{0}(t), \\
z_{i}(0)=&~x,
\end{aligned}%
\right.
\end{equation}
and the limiting cost functional
\begin{equation}\label{lcost}
\begin{aligned}
J_i(u_{i}(\cdot )) =&\frac{1}{2}\mathbb{E}\Big\{\int_{0}^{T}\big[\langle
Q(t)(z_{i}(t)-m(t)),z_{i}(t)-m(t)\rangle +\langle
R(t)u_{i}(t),u_{i}(t)\rangle \big]dt +\langle Gz_{i}(T),z_{i}(T)\rangle \Big\}.
\end{aligned}
\end{equation}
Then the corresponding auxiliary limiting problem with incomplete information is proposed as follows.

Now, we formulate the following limiting stochastic control problem with incomplete information.\\
\textbf{Problem (\uppercase\expandafter{\romannumeral2}).}
For $1\leq i\leq N$,
find $\bar{u}_i(\cdot)\in \mathcal{U}_{i}$ satisfying
\begin{equation*}
J_{i}(\bar{u}_{i}(\cdot ))=\underset{
u_{i}\left( \cdot \right) \in \mathcal{U}_{i}}{\inf }J
_{i}(u_{i}(\cdot )).
\end{equation*}
Then $\bar{u}_i(\cdot)$ is called the decentralized optimal control for \textbf{Problem (\uppercase\expandafter{\romannumeral2})}. Moreover, the corresponding state $\bar{z}_i(\cdot)$ is called the decentralized optimal trajectory.

\begin{remark}\label{remark for m}
Compared with \cite{LNW2022}, where the state-average limit is deterministic, the counterpart $m(\cdot)$ in our setting turns out to be a \emph{stochastic process} which is $\mathcal{F}_t^{W_0}$-adapted. In fact, we will see later that $m(\cdot)$ solves a SDE driven by the common noise, see \eqref{m-1}, which yields that $m(\cdot)\in L^2_{\mathcal{F}^{W_0}}(0,T;\mathbb{R}^n)$.
\end{remark}
\subsection{Open-loop decentralized strategies}\label{sec2.1}
Noting that with the help of the frozen limiting state-average, the \textbf{Problem (\uppercase\expandafter{\romannumeral2})} is essentially a LQ stochastic control problem with incomplete information.
To this end, we will apply the stochastic maximum principle with partial information to obtain optimal control.

Let us define the Hamiltonian function $H:[0,T]\times\mathbb{R}^n\times\mathbb{R}^n\times\mathbb{R}^n\times\mathbb{R}^n\times \mathbb{R}^m\rightarrow\mathbb{R}$ as
\begin{equation}\label{Hamiltonian}
\begin{aligned}
&H(t,p_{i},q_{i},q_{i,0},z_{i},u_{i}) \\
=&\langle
p_{i},Az_{i}+Bu_{i}+\alpha m+b\rangle +\langle q_{i},Cz_{i}+Du_{i}+\beta m+\sigma
\rangle \\
&+\langle q_{i,0},C_0z_{i}+D_0u_{i}+
\beta_0m+\sigma_0\rangle -\frac{1}{2}\langle Q(z_{i}-m),z_{i}-m\rangle -\frac{1}{2}\langle Ru_{i},u_{i}\rangle.
\end{aligned}
\end{equation}
and we introduce the following adjoint equation
\begin{equation}\label{adjoint}
\left\{
\begin{aligned}
dp_{i}(t)=&-\Big[A^{\top }(t)p_{i}(t)+C^{\top}(t)q_{i}\left( t\right) +
C_0^{\top }(t)q_{i,0}\left( t\right)-Q(t)\big(z_i(t) -m(t)\big)\Big]dt\\
 &+q_{i}\left( t\right) dW_{i}(t)+q_{i,0}\left( t\right) dW_{0}(t), \\
p_{i}(T)=&-Gz_{i}\left( T\right).
\end{aligned}
\right.
\end{equation}
Note that the solution $(p_i(\cdot),q_i(\cdot),q_{i,0}(\cdot))$  of \eqref{adjoint} belongs to the space $L^2_{\mathcal{F}^i}(0,T;\mathbb{R}^n)\times L^2_{\mathcal{F}^i}(0,T;\mathbb{R}^n)\times L^2_{\mathcal{F}^i}(0,T;\mathbb{R}^n)$.

By applying the stochastic maximum principle, we can obtain the open-loop decentralized strategies for \textbf{Problem (\uppercase\expandafter{\romannumeral2})}.

\begin{theorem}
For $1\leq i\leq N$, let \textup{(A1)}-\textup{(A2)} hold. Then the decentralized optimal control of \textbf{Problem (\uppercase\expandafter{\romannumeral2})} is given by
\begin{equation}\label{gcontrol}
\begin{aligned}
\bar{u}_{i}(t)&=R(t)^{-1}\big(B^{\top }(t)\mathbb{E}[\bar{p}
_{i}(t)|\mathcal{F}_t^{W_i}]+D^{\top }(t)\mathbb{E}[\bar{q}_{i}(t)|\mathcal{F}_t^{W_i}]+D_0^{\top }(t)\mathbb{E}[\bar{q}_{i,0}(t)|\mathcal{F}_t^{W_i}]\big),
\end{aligned}
\end{equation}
where $(\bar{z}_i(\cdot),\bar{p}_i(\cdot),\bar{q}_i(\cdot),\bar{q}_{i,0}(\cdot))$ satisfy the following stochastic Hamiltonian system
\begin{equation}\label{H system}
\left\{
\begin{aligned}
d\bar{z}_{i}(t)=&\Big\{A(t)\bar{z}_{i}(t)+B(t)R^{-1}(t)(B^{\top }(t)\mathbb{E}[\bar{p}_{i}(t)|\mathcal{F}_t^{W_i}]+D^{\top }(t)\mathbb{E}[\bar{q}_{i}(t)|\mathcal{F}_t^{W_i}]\\
&+D_0^{\top}(t)\mathbb{E}[\bar{q}_{i,0}(t)|\mathcal{F}_t^{W_i}])+\alpha(t) m(t)+b(t)\Big\}dt\\
&+\Big\{C(t)\bar{z}_{i}(t)+D(t)R^{-1}(t)(B^{\top }(t)\mathbb{E}[\bar{p}_{i}(t)|
\mathcal{F}_t^{W_i}]+D^{\top }(t)\mathbb{E}[\bar{q}_{i}(t)|\mathcal{F}_t^{W_i}]\\
&+D_0^{\top }(t)\mathbb{E}[\bar{q}_{i,0}(t)|\mathcal{F}_t^{W_i}])+\beta(t)m(t)+\sigma (t)\Big\}dW_{i}(t)\\
&+\Big\{C_0(t)\bar{z}_{i}(t)+D_0(t)R^{-1}(t)(B^{\top }(t)\mathbb{E}[\bar{p
}_{i}(t)|\mathcal{F}_t^{W_i}]+D^{\top }(t)\mathbb{E}[\bar{q}_{i}(t)|
\mathcal{F}_t^{W_i}]\\
&+D_0^{\top }(t)\mathbb{E}[\bar{q}_{i,0}(t)|\mathcal{F}_t^{W_i}])+\beta_0(t)m(t)+\sigma_0(t)\Big\}dW_{0}(t), \\
d\bar{p}_{i}(t)=&-\Big[A^{\top }(t)\bar{p}_{i}(t)+C^{\top}(t)\bar{q}_{i}\left( t\right) +
C_0^{\top }(t)\bar{q}_{i,0}\left( t\right) \\
&-Q(t)\big(\bar{z}_i(t)-m(t)\big)\Big]dt +\bar{q}_{i}\left( t\right) dW_{i}(t)+\bar{q}_{i,0}\left( t\right) dW_{0}(t), \\
\bar{z}_{i}(0)=&x,\quad \bar{p}_{i}(T)=-G\bar{z}_{i}\left( T\right).
\end{aligned}
\right.
\end{equation}
\end{theorem}

The proof is simple. In fact, the results follows by the application of the stochastic maximum principle.

\begin{proof}
For the decentralized optimal control $\bar{u}_{i}(\cdot)$, suppose that $\bar{z}_i(\cdot)$ is corresponding state trajectory, and $(\bar{p}_i(\cdot),\bar{q}_i(\cdot),\bar{q}_{i,0}(\cdot))$ is the unique solution to the second equation of \eqref{H system} with respect to $(\bar{z}_{i}(\cdot),\bar{u}_{i}(\cdot))$, the maximum principle reads as the following form
\begin{equation}\label{SMP1}
\begin{aligned}
&\mathbb{E}\Big[\Big\langle \frac{\partial H}{\partial u_{i}}(t,\bar{p}_{i}(t),\bar{q}_{i}(t),
\bar{q}_{i,0}(t),\bar{z}_{i}(t),\bar{u}_{i}(t)),u_i-\bar{u}_{i}(t)\Big\rangle \Big|\mathcal{F}
_{t}^{W_i}\Big]= 0,
\quad \text{for any}\ u_i\in \mathbb{R}^k,\ a.s.\ a.e.
\end{aligned}
\end{equation}
which yields that (by noticing \eqref{Hamiltonian})
\begin{equation*}
\begin{aligned}
\Big\langle &\mathbb{E}[B^{\top }(t)\bar{p}_{i}(t)+D^{\top }(t)\bar{q}_{i}(t)+D_0^{\top
}(t)\bar{q}_{i,0}(t)-R(t)\bar{u}_{i}(t)|\mathcal{F}_{t}^{W_i}],
u_i-\bar{u}_{i}(t)\Big\rangle = 0,
\end{aligned}
\end{equation*}
and by noticing
$u_i$ is an arbitrary element of $\mathbb{R}^k$, we have
\begin{equation*}
\begin{aligned}
&\mathbb{E}[B^{\top }(t)\bar{p}_{i}(t)+D^{\top }(t)\bar{q}_{i}(t)+D_0^{\top
}(t)\bar{q}_{i,0}(t)-R(t)\bar{u}_{i}(t)|\mathcal{F}
_{t}^{W_i}]= 0.
\end{aligned}
\end{equation*}
Then from $R(\cdot)\gg0$, we obtain \eqref{gcontrol}.
\end{proof}

\begin{remark}
Due to the convexity of the admissible control set, here we only need the first-order adjoint equation. Moroever, since the cost functional of \textbf{Problem (\uppercase\expandafter{\romannumeral2})} is strictly convex, it admits a unique optimal control, thus the sufficiency of optimal control can also be obtained. Furthermore, a feedback optimal control can be obtained via Riccati equation by applying stochastic maximum principle instead of using dynamic programming principle (DPP) and Hamilton-Jacobi-Bellman (HJB) equation. We mention that it will have some difficulties to obtain feedback optimal control via DPP, in fact feedback optimal control can be given through solving related HJB equations, however due to the randomness of the  coefficients of \eqref{z-lstate} (recall that $m$ is an $\mathcal{F}_t^{W_0}$-adapted process), the related HJB equation will be stochastic PDE whose solvability is challenging.
\end{remark}

Now let us study the unknown frozen limiting state-average $m(\cdot)$. When $N\rightarrow\infty$, we would like to approximate $\bar{x}_i(\cdot)$ by $\bar{z}_i(\cdot)$ (see \eqref{CC system-1}), thus $\frac{1}{N}\sum_{i=1}^N \bar{x}_{i}(\cdot)$ is approximated by $\frac{1}{N}\sum_{i=1}^N \bar{z}_{i}(\cdot)$. Moreover, inspired by \cite{HHN2018},
recall that for $i\neq j$, $\bar{z}_i$ and $\bar{z}_j$ are identically distributed and conditional independent (under $\mathbb{E}[\cdot|\mathcal{F}_{\cdot}^{W_0}]$),
by the conditional strong law of large number, it follows that (the convergence is in the sense of almost surely, see \cite{MNZ2005})
\begin{equation}\label{limit}
m(\cdot)=\underset{N\rightarrow \infty }{\lim }\frac{1}{N}\overset{N}{\underset{
i=1}{\sum }}\bar{z}_{i}(\cdot )=\mathbb{E}[\bar{z}_{i}(\cdot)|\mathcal{F}_{\cdot}^{W_0}].
\end{equation}

From \eqref{z-lstate}, by replacing $m(\cdot)$ by $\mathbb{E}[\bar{z}_{i}(\cdot)|\mathcal{F}_{\cdot}^{W_0}]$, we obtain
\begin{equation}\label{s-lstate}
  \left\{
\begin{aligned}
d\bar{z}_{i}&(t)=\{A(t)\bar{z}_{i}(t)+B(t)\bar{u}_i(t)+\alpha(t) \mathbb{E}[\bar{z}_{i}(t)|\mathcal{F}_t^{W_0}]+b(t)\}dt\\
&+\{C(t)\bar{z}_{i}(t)+D(t)\bar{u}_i(t)+\beta(t) \mathbb{E}[\bar{z}_{i}(t)|\mathcal{F}_t^{W_0}]+\sigma (t)\}dW_{i}(t)\\
&+\{C_0(t)\bar{z}_{i}(t)+D_0(t)\bar{u}_i(t)+\beta_0(t) \mathbb{E}[\bar{z}_{i}(t)|\mathcal{F}_t^{W_0}]+\sigma_0(t)\}dW_{0}(t), \\
\bar{z}_{i}&(0)=x.
\end{aligned}
\right.
\end{equation}
By taking expectation on both side of \eqref{s-lstate}, we obtain
\begin{equation}\label{Ez-equation}
  \left\{
\begin{aligned}
d\mathbb{E}[\bar{z}_{i}(t)]=&\{(A(t)+\alpha(t))\mathbb{E}[\bar{z}_{i}(t)]
+B(t)\mathbb{E}[\bar{u}_i(t)]+b(t)\}dt,\\
\mathbb{E}[\bar{z}_{i}(0)]=&x.
\end{aligned}
\right.
\end{equation}
Recall that the decentralized admissible control
$\bar{u}_i$ is $\{\mathcal{F}_t^{W_i}\}_{0\leq t\leq T}$-adapted and $W_i(\cdot)$ and $W_0(\cdot)$ are independent, we have $\mathbb{E}[\bar{u}_i(t)|\mathcal{F}_t^{W_0}]=\mathbb{E}[\bar{u}_i(t)]$. Thus, by taking conditional expectation w.r.t. $\mathcal{F}_t^{W_0}$ on both side of \eqref{s-lstate},  we have (recalling $m(\cdot)=\mathbb{E}[\bar{z}_i(\cdot)|\mathcal{F}_\cdot^{W_0}]$)
\begin{equation}\label{m-1}
  \left\{
\begin{aligned}
&dm(t)=\{(A(t)+\alpha(t))m(t)+B(t)\mathbb{E}[\bar{u}_i(t)]+b(t)\}dt\\
&+\{(C_0(t)+\beta_0(t))m(t)+D_0(t)\mathbb{E}[\bar{u}_i(t)]+\sigma_0(t)\}dW_{0}(t), \\
&m(0)=x,
\end{aligned}
\right.
\end{equation}
which implies that
\begin{equation}\label{Em-1}
  \left\{
\begin{aligned}
d\mathbb{E}[m(t)]=&\{(A(t)+\alpha(t))\mathbb{E}[m(t)]
+B(t)\mathbb{E}[\bar{u}_i(t)]+b(t)\}dt,\\
\mathbb{E}[m(0)]=&x.
\end{aligned}
\right.
\end{equation}
Now, we are going back to  \eqref{H system}, by replacing $m(\cdot)$ with $\mathbb{E}[\bar{z}_{i}(\cdot)|\mathcal{F}^{W_0}]$, we deduce the following consistency condition system (or Nash certainty equivalence equation system, see, e.g. \cite{HCM2006, HMC2006}) which is a mean-field forward backward stochastic differential equation (MF-FBSDE)
\begin{equation}\label{CC system-1}
\left\{
\begin{aligned}
d\bar{z}_{i}(t)=&\{A(t)\bar{z}_{i}(t)+B(t)R^{-1}(t)(B^{\top }(t)\mathbb{E}[\bar{p}_{i}(t)|\mathcal{F}_{t}^{W_i}]+D^{\top}(t)\mathbb{E}[\bar{q}_{i}(t)|\mathcal{F}_{t}^{W_i}]\\
&+D_0^{\top}(t)\mathbb{E}[\bar{q}_{i,0}(t)|\mathcal{F}
_{t}^{W_i}])+\alpha(t) \mathbb{E}[\bar{z}_{i}(t)|\mathcal{F}_{t}^{W_0}]+b(t)\}dt\\
&+\{C(t)\bar{z}_{i}(t)+D(t)R^{-1}(t)(B^{\top }(t)\mathbb{E}[\bar{p}_{i}(t)|
\mathcal{F}_{t}^{W_i}]+D^{\top }(t)\mathbb{E}[\bar{q}_{i}(t)|\mathcal{F}_{t}^{W_i}]\\
&+D_0^{\top }(t)\mathbb{E}[\bar{q}_{i,0}(t)|\mathcal{F}_{t}^{W_i}])+\beta(t) \mathbb{E}[\bar{z}_{i}(t)|\mathcal{F}_{t}^{W_0}]+\sigma (t)\}dW_{i}(t)\\
&+\{C_0(t)\bar{z}_{i}(t)+D_0(t)R^{-1}(t)(B^{\top }(t)\mathbb{E}[\bar{p
}_{i}(t)|\mathcal{F}_{t}^{W_i}]+D^{\top }(t)\mathbb{E}[\bar{q}_{i}(t)|
\mathcal{F}_{t}^{W_i}]\\
&+D_0^{\top }(t)\mathbb{E}[\bar{q}_{i,0}(t)|\mathcal{F}_{t}^{W_i}])
+\beta_0(t) \mathbb{E}[\bar{z}_{i}(t)|\mathcal{F}_{t}^{W_0}]+\sigma_0(t)\}dW_{0}(t), \\
d\bar{p}_{i}(t)=&-[A^{\top }(t)\bar{p}_{i}(t)+C^{\top}(t)\bar{q}_{i}\left( t\right) +
C_0^{\top }(t)\bar{q}_{i,0}\left( t\right) -Q(t)(\bar{z}_i\left(
t\right) -\mathbb{E}[\bar{z}_{i}(t)|\mathcal{F}_{t}^{W_0}])]dt\\
&+\bar{q}_{i}\left( t\right) dW_{i}(t)+\bar{q}_{i,0}\left( t\right) dW_{0}(t), \\
\bar{z}_{i}(0)=&x,\quad \bar{p}_{i}(T)=-G\bar{z}_{i}\left( T\right).
\end{aligned}
\right.
\end{equation}

Now let us consider MF-FBSDE \eqref{CC system-1}, which has fully coupled structure and conditional expectation terms. There are difficulties for establishing its well-posedness due to these features. By the discounting method, Hu, Huang and Nie \cite{HHN2018} first proved the well-posedness of a kind of MF-FBSDE with conditional expectation. The follow-up work \cite{LNW2022} extended \cite{HHN2018} to more general case. Using Theorem 3.3 in \cite{LNW2022}, we can obtain the following well-posedness of MF-FBSDE \eqref{CC system-1}.
\begin{theorem}
%Suppose that assumptions \textup{(A1)} and \textup{(A2)} hold.
Let $\lambda^{\ast}$ be the maximum eigenvalue of matrix $\frac{A+A^{\top}}{2}$, suppose that $4\lambda^{\ast }<-2|\alpha|-6|C|^{2}-6|C_0|^{2}-5|\beta|^{2}-5|\beta_0|^{2}$, there exists a constant $\theta_1>0$ , which depends on $\lambda^{\ast}$, $|C|$, $|C_0|$, $|\alpha|$, $|\beta|$, $|\beta_0|$, $|Q|$, $|G|$, and does not depend on $T$, such that when $|B|$, $|D|$, $|D_0|$ and $|R^{-1}|\in[0,\theta_1)$,  then there exists a unique adapted solution $(\bar{z}_i(\cdot),\bar{p}_i(\cdot),\bar{q}_i(\cdot),\bar{q}_{i,0}(\cdot))\in $ $L^2_{\mathcal{F}^i}([0,T];\mathbb{R}^n\times \mathbb{R}^n\times \mathbb{R}^n\times \mathbb{R}^n)$ to MF-FBSDE \eqref{CC system-1}.
\end{theorem}

\begin{remark}
We mention that the consistency condition system \eqref{CC system-1} is a fully coupled conditional MF-FBSDE, which has the conditional expectation terms $\mathbb{E}[\bar{z}_i(\cdot)|\mathcal{F}^{W_0}]$, $\mathbb{E}[\bar{p}_i(\cdot)|\mathcal{F}^{W_i}]$,  $\mathbb{E}[\bar{q}_i(\cdot)|\mathcal{F}^{W_i}]$ and $\mathbb{E}[\bar{q}_{i,0}(\cdot)|\mathcal{F}^{W_i}]$.
In comparison with our current work, the MF-FBSDE in \cite{LNW2022} includes the expectation term $\mathbb{E}[\bar{z}_i(\cdot)]$. Even through our MF-FBSDE \eqref{CC system-1} includes the conditional expectation term $\mathbb{E}[\cdot|\mathcal{F}^{W_0}]$, the approach for proving the well-posedness of \eqref{CC system-1} is the same as we discussed in \cite{LNW2022}.
\end{remark}

\subsection{$\varepsilon$-Nash Equilibrium Analysis for Problem \textbf{(\uppercase\expandafter{\romannumeral1})}}\label{sec2.2}

In subsection \ref{sec2.1}, we obtained the optimal strategy profile $\bar{u}(\cdot)=(\bar{u}_1(\cdot),\bar{u}_2(\cdot),\ldots,\bar{u}_N(\cdot))$ of \textbf{Problem (\uppercase\expandafter{\romannumeral2})}, see Theorem 2.1 and \eqref{gcontrol}. In this section, we will show  $\bar{u}(\cdot)=(\bar{u}_1(\cdot),\bar{u}_2(\cdot),\ldots,\bar{u}_N(\cdot))$ is an $\varepsilon$-Nash equilibrium for \textbf{Problem (\uppercase\expandafter{\romannumeral1})}. Firstly, we give the definition of $\varepsilon$-Nash equilibrium.

\begin{definition}\label{varepsilon Nash equilibrium}
A set of controls $u_i(\cdot)\in \mathcal{U}_{ad}^{c}$, $1\leq i \leq N$, for $N$ agents is called an $\varepsilon$-Nash equilibrium with respect to the cost $\mathcal{J}_i$, $1\leq i \leq N$, if there exists $\varepsilon=\varepsilon(N)\geq 0$ and $\underset{N\rightarrow\infty}{\lim }\varepsilon(N)=0$ such that for any $1\leq i \leq N$, we have
\begin{equation*}
\mathcal J_{i}(\bar{u}_{i}(\cdot ),\bar{u}_{-i}(\cdot ))\leq \mathcal{J}_{i}(u_{i}(\cdot),\bar{u}_{-i}(\cdot))+\varepsilon,
\end{equation*}
when any alternative strategy $u_{i}(\cdot )\in \mathcal{U}_{ad}^{c}$ is applied by $\mathcal{A}_i$.
\end{definition}

\begin{remark}
  If $\varepsilon=0$, Definition \rm{\ref{varepsilon Nash equilibrium}} can reduce to the usual exact Nash equilibrium.
\end{remark}

Now, we give the following main result of this section and its proof will be given later.
\begin{theorem}\label{main theorem}
  Under \textup{(A1)} and \textup{(A2)}, the strategy set $(\bar{u}_1(\cdot),\bar{u}_2(\cdot),\ldots,\bar{u}_N(\cdot))$ is an $\varepsilon$-Nash equilibrium of \textbf{Problem (\uppercase\expandafter{\romannumeral1})}, where $\bar{u}_i(\cdot)$, $1\leq i\leq N$, is given by \eqref{gcontrol}.
\end{theorem}

In order to prove the above theorem, let us give several lemmas.

\begin{lemma}\label{lemma est 1}
Let \textup{(A1)}-\textup{(A2)} hold,  it follows that
%\begin{equation}\label{est 1}
%\mathbb{E}\underset{0\leq t\leq T}{\sup}|\hat{\bar{z}}^{(N)}(t)-\mathbb{E}[m(t)]|^{2}=O(\frac{1}{N}),
%\end{equation}
\begin{equation}\label{est 2}
\mathbb{E}\underset{0\leq t\leq T}{\sup }|\bar{x}^{(N)}(t)-m(t)|^{2}=O(\frac{1}{N}),
\end{equation}
\begin{equation}\label{est 3}
\underset{1\leq i\leq N}{\sup }\mathbb{E}\underset{0\leq t\leq T}{\sup}|
\bar{x}_{i}(t)-\bar{z}_{i}(t)|^{2}=O(\frac{1}{N}).
\end{equation}
\end{lemma}

\begin{proof}
From \eqref{state} and \eqref{m-1},
by using Cauchy-Schwarz inequality and Burkholder-Davis-Gundy (BDG) inequality, it follows that there exists a constant $K$ (independent of $N$, which may vary line to line in the following)
\begin{equation*}
  \begin{aligned}
  \mathbb{E}&\underset{0\leq s\leq t}{\sup}|\bar{x}^{(N)}(s)-m(s)|^{2}
 \leq
  K \mathbb{E}\int_{0}^{t}\big|\bar{x}^{(N)}(s)-m(s)\big|^2ds  +K\mathbb{E}\int_{0}^{t}\big|\frac{1}{N}\sum_{i=1}^{N}\big(\bar{u}_i(t)-\mathbb{E}[\bar{u}_i(t)]\big)\big|^2dt\\
  &+\frac{K}{N^{2}}\mathbb{E}\overset{N}{\underset{i=1}{\sum }}
\int_{0}^{T}|C(t)\bar{x}_{i}(t)+D(t)\bar{u}_{i}(t)+\alpha(t)(\bar{x}^{(N)}(t)-m(t))+\alpha(t)m(t)+\sigma (t)|^{2}dt \\
&+\frac{K}{N^{2}}\mathbb{E}\overset{N}{\underset{i=1}{\sum }}\int_{0}^{T}|C_0(t)\bar{x}_{i}(t)+D_0(t)\bar{u}_{i}(t)+\beta(t)(\bar{x}^{(N)}(t)-m(t))
+\beta_0(t)m(t)+\sigma_0(t)|^{2}dt.
  \end{aligned}
\end{equation*}
Therefore, noticing that $\mathbb{E}\sup\limits_{0\leq t\leq T}|\bar{x}_i(t)|^2<+\infty$ and $\mathbb{E}\sup\limits_{0\leq t\leq T}|m(t)|^2\leq K$,
we obtain \eqref{est 2} by Gronwall's inequality. Then,
from \eqref{state}, \eqref{z-lstate} and \eqref{est 2}, we can show \eqref{est 3} by Gronwall's inequality.
\end{proof}

\begin{lemma}\label{difference of cf 1}
 For $1\leq i \leq N$, we have
\begin{equation*}
|\mathcal{J}_{i}(\bar{u}_{i}(\cdot ),\bar{u}_{-i}(\cdot ))-J_{i}(\bar{u}%
_{i}(\cdot ))|=O\Big(\frac{1}{\sqrt{N}}\Big).
\end{equation*}
\end{lemma}
\begin{proof}
From \eqref{cost} and \eqref{lcost}, we have
\begin{equation*}
\begin{aligned}
&\mathcal{J}_{i}(\bar{u}_{i}(\cdot ),\bar{u}_{-i}(\cdot))-J_{i}(\bar{u}_{i}(\cdot))\\
 =&\frac{1}{2}\mathbb{E}\Big\{\int_{0}^{T}\Big[\big\langle Q(t)\big(\bar{x}
_{i}(t)-\bar{x}^{(N)}(t)\big),\bar{x}_{i}(t)-\bar{x}^{(N)}(t)\big\rangle-\big\langle Q(t)\big(\bar{z}_{i}(t)-m(t)\big),\bar{z}_{i}(t)-m(t)\big\rangle \Big]dt\\
&+\big\langle G\bar{x}_{i}(T),\bar{x}_{i}(T)\big\rangle-\big\langle G\bar{z}_{i}(T),\bar{z}_{i}(T)\big\rangle\Big\}.
\end{aligned}
\end{equation*}

By noticing that Lemma \ref{lemma est 1}, $\mathbb{E}\sup\limits_{0\leq t\leq T}|\bar{z}_i(t)|^2\leq K$ and $\mathbb{E}\sup\limits_{0\leq t\leq T}|m(t)|^2\leq K$, for some constant $K$ independent of $N$, we have (using $\langle Qa,a\rangle-\langle Qb,b\rangle=\langle Q(a-b),a-b\rangle+2\langle Q(a-b),b\rangle$)
\begin{equation*}
\begin{aligned}
&\Big|\mathbb{E}\int_{0}^{T}\Big[\big\langle Q(t)(\bar{x}_{i}(t)-\bar{x}^{(N)}(t)),
\bar{x}_{i}(t)-\bar{x}^{(N)}(t)\rangle -\langle Q(t)(\bar{z}
_{i}(t)-m(t)),\bar{z}_{i}(t)-m(t)\big\rangle \Big]dt\Big| \\
%\leq &K\int_{0}^{T}\mathbb{E}\big|\bar{x}_{i}(t)-\bar{z}_{i}(t)\big|^{2}dt+K
%\int_{0}^{T}\mathbb{E}\big|\bar{x}^{(N)}(t)-m(t)\big|^{2}dt\\&
%+K\int_{0}^{T}\big(\mathbb{E}\big|
%\bar{x}_{i}(t)-\bar{x}^{(N)}(t)-(\bar{z}_{i}(t)-m(t)\big)\big|^{2})^{\frac{1}{2}
%}\big(\mathbb{E}\big|\bar{z}_{i}(t)-m(t)\big|^{2}\big)^{\frac{1}{2}}dt \\
\leq &K\int_{0}^{T}\mathbb{E}\big|\bar{x}_{i}(t)-\bar{z}_{i}(t)\big|^{2}dt+K
\int_{0}^{T}\mathbb{E}\big|\bar{x}^{(N)}(t)-m(t)\big|^{2}dt\\
&+K\int_{0}^{T}\big(\mathbb{E}\big|
\bar{x}_{i}(t)-\bar{z}_{i}(t)\big|^{2}+\mathbb{E}\big|\bar{x}^{(N)}(t)-m(t)\big|^{2}\big)^{
\frac{1}{2}}dt=O\Big(\frac{1}{\sqrt{N}}\Big).
\end{aligned}
\end{equation*}%
Similarly, we can also prove that the order of terminal term is $O\Big(\frac{1}{\sqrt{N}}\Big)$, which completes the proof.
\end{proof}

Next, let us consider a perturbed control $u_i(\cdot)$ for $\mathcal{A}_i$, the corresponding state system is, for $1\leq i\leq N$,
\begin{equation}\label{perturbed state yi}
\left\{
\begin{aligned}
dy_{i}(t)=&\big[A(t)y_{i}(t)+B(t)u_{i}(t)+\alpha(t) y^{(N)}(t)+b(t)\big]dt \\
&+\big[C(t)y_{i}(t)+D(t)u_{i}(t)+\beta(t) y^{(N)}(t)+\sigma(t)\big]dW_{i}(t)\\
&+\big[C_0(t)y_{i}(t)+D_0(t)u_{i}(t)+\beta_0(t) y^{(N)}(t)+\sigma_0(t)\big]dW_{0}(t),\\
y_{i}(0)=&x,
\end{aligned}
\right.
\end{equation}
whereas other agents keep the control $\bar{u}_j(\cdot)$, $1\leq i\leq N$, $j\neq i$, and have the following state dynamics
\begin{equation}\label{perturbed state yj}
\left\{
\begin{aligned}
dy_{j}(t)=&\big[A(t)y_{j}(t)+B(t)\bar{u}_{j}(t)+\alpha(t) y^{(N)}(t)+b(t)\big]dt \\
&+\big[C(t)y_{j}(t)+D(t)\bar{u}_{j}(t)+\beta(t) y^{(N)}(t)+\sigma(t)\big]dW_{j}(t)\\
&+\big[C_0(t)y_{j}(t)+D_0(t)\bar{u}_{j}(t)+\beta_0(t) y^{(N)}(t)+\sigma_0(t)\big]dW_{0}(t),\\
y_{j}(0)=&x,
\end{aligned}
\right.
\end{equation}
where $y^{(N)}(\cdot)=\frac{1}{N}\sum\limits_{i=1}^N y_{i}(\cdot)$.
To prove that $(\bar{u}_1(\cdot),\ldots,\bar{u}_N(\cdot))$ is an $\varepsilon$-Nash equilibrium, we need to show
\begin{equation*}
\underset{u_{i}\left( \cdot \right) \in \mathcal{U}_{ad}^{c}}{\inf }\mathcal{J
}_{i}(u_{i}(\cdot ),\bar{u}_{-i}(\cdot ))
\geq \mathcal{J}_{i}(\bar{u}_{i}(\cdot ),\bar{u}_{-i}(\cdot ))-\varepsilon .
\end{equation*}
Then it only needs to consider the perturbation $\mathcal{U}_{ad}^{c}$ such that
$\mathcal{J}_{i}(u_{i}(\cdot ),\bar{u}_{-i}(\cdot )) \leq\mathcal{J}_{i}(\bar{u}_{i}(\cdot ),\bar{u}_{-i}(\cdot ))$.
Then
\begin{equation*}
\begin{aligned}
&\mathbb{E}\int_{0}^{T}\langle R(t)u_{i}(t),u_{i}(t)\rangle dt\leq \mathcal{J}
_{i}(u_{i}(\cdot ),\bar{u}_{-i}(\cdot ))\leq \mathcal{J}_{i}(\bar{u}_{i}(\cdot ),
\bar{u}_{-i}(\cdot ))
= J_{i}(\bar{u}_{i}(\cdot ))+O(\frac{1}{\sqrt{N}}),
\end{aligned}
\end{equation*}
which yields that $\mathbb{E}\int_{0}^{T}|u_{i}(t)|^{2}dt\leq K$.

For the $i$-th agent, we consider the limiting state with perturbation control
\begin{equation}\label{bar yi}
  \left\{
  \begin{aligned}
  d\bar{y}_i(t)=&\big[A(t)\bar{y}_i(t)+B(t)u_{i}(t)+\alpha(t)m(t)+b(t)\big]dt \\
&+\big[C(t)\bar{y}_i(t)+D(t)u_{i}(t)+\beta(t)m(t)+\sigma(t)\big]dW_{i}(t)\\
&+\big[C_0(t)\bar{y}_i(t)+D_0(t)u_{i}(t)+\beta_0(t)m(t)+\sigma_0(t)\big]dW_{0}(t),\\
\bar{y}_{i}(0)=&x.
  \end{aligned}
  \right.
\end{equation}

Then, we have the following estimates.
\begin{lemma}\label{lemma est 2}
Let \textup{(A1)}-\textup{(A2)} hold, then
\begin{equation}\label{est 4}
\mathbb{E}\underset{0\leq t\leq T}{\sup }|y^{(N)}(t)-m(t)|^{2}=O(\frac{1}{N}),
\end{equation}
\begin{equation}\label{est 5}
\underset{1\leq i\leq N}{\sup }\mathbb{E}\underset{0\leq t\leq T}{\sup }
|y_{i}(t)-\bar{y}_{i}(t)|^{2}=O(\frac{1}{N}).
\end{equation}
\end{lemma}

\begin{proof}
Firstly, we prove \eqref{est 4}.
%and we denote $y^{(N)}(t)=\frac{1}{N}\sum\limits_{j=1}^{N}y_j(t)$.
%\hat{\bar{y}}^{(N)}(t)=&\frac{1}{N}\sum_{j=1}^{N}\hat{\bar{y}}_j(t),\
%y^{(N-1)}(t)=&\frac{1}{N-1}\sum_{j=1,j\neq i}^{N}y_j(t),
%\hat{\bar{y}}^{(N-1)}(t)=&\frac{1}{N-1}\sum_{j=1,j\neq i}^{N}\hat{\bar{y}}_j(t).
From \eqref{perturbed state yi}, \eqref{perturbed state yj}, \eqref{m-1}
and BDG inequality, we get
\begin{equation*}
\begin{aligned}
  \mathbb{E}&\underset{0\leq s\leq t}{\sup}|y^{(N)}(s)-m(s)|^{2} \leq
  K \mathbb{E}\int_{0}^{t}\big|y^{(N)}(s)-m(s)\big|^2ds+\frac{K}{N^2}\mathbb{E}\int_{0}^{t}|u_i(s)|^2ds\\
  &+K \mathbb{E}\int_{0}^{t}\big|\frac{1}{N}\sum_{j=1,j\neq i}^{N}\bar{u}_j(s)-\mathbb{E}[\bar{u}_i(s)]\big|^2ds+\frac{K}{N^2}\sum_{j=1}^{N}\mathbb{E}\int_{0}^{t}|C(t)y_j(s)+\beta(t)(y^{(N)}(t)-m(t))\\
 &+\beta(t)m(t)+\sigma(s)|^2ds
 +\frac{K}{N^2}\sum_{j=1}^{N}\mathbb{E}\int_{0}^{t}|C_0(t)y_j(s)+\beta_0(t)(y^{(N)}(t)-m(t))\\
 &+\beta_0(t)m(t)+\sigma_0(s)|^2ds
 +\frac{K}{N^2}\sum_{j=1,j\neq i}^{N}\mathbb{E}\int_{0}^{t}|\bar{u}_j(s)|^2ds.\\
\end{aligned}
\end{equation*}

Moreover, taking conditional expectation w.r.t. $\mathcal{F}_t^{W_i}$ on both side of \eqref{H system}, we have
\begin{equation}\label{H system-1}
\left\{
\begin{aligned}
d\hat{\bar{z}}_{i}(t)=&\Big\{A(t)\hat{\bar{z}}_{i}(t)+B(t)R^{-1}(t)\big(B^{\top }(t)\mathbb{E}[\bar{p}_{i}(t)|\mathcal{F}_t^{W_i}]+D^{\top }(t)\mathbb{E}[\bar{q}_{i}(t)|\mathcal{F}_t^{W_i}]\\
&+D_0^{\top}(t)\mathbb{E}[\bar{q}_{i,0}(t)|\mathcal{F}_t^{W_i}]\big)+\alpha(t) \mathbb{E}[m(t)]+b(t)\Big\}dt\\
&+\Big\{C(t)\hat{\bar{z}}_{i}(t)+D(t)R^{-1}(t)\big(B^{\top }(t)\mathbb{E}[\bar{p}_{i}(t)|
\mathcal{F}_t^{W_i}]+D^{\top }(t)\mathbb{E}[\bar{q}_{i}(t)|\mathcal{F}_t^{W_i}]\\
&+D_0^{\top }(t)\mathbb{E}[\bar{q}_{i,0}(t)|\mathcal{F}_t^{W_i}]\big)+\beta(t)\mathbb{E}[m(t)]+\sigma (t)\Big\}dW_{i}(t)\\
d\hat{\bar{p}}_{i}(t)=&-\Big[A^{\top }(t)\hat{\bar{p}}_{i}(t)+C^{\top}(t)\hat{\bar{q}}_{i}\left( t\right) +
C_0^{\top }(t)\hat{\bar{q}}_{i,0}\left( t\right) -Q(t)\big(\hat{\bar{z}}_i(t)\\
&-\mathbb{E}[m(t)]\big)\Big]dt
+\hat{\bar{q}}_{i}\left( t\right) dW_{i}(t), \\
\hat{\bar{z}}_{i}(0)=&x,\quad \hat{\bar{p}}_{i}(T)=-G\hat{\bar{z}}_{i}\left( T\right),
\end{aligned}
\right.
\end{equation}
then from \eqref{H system-1} and \eqref{gcontrol}, we obtain that $\{\bar{u}_i(t)\}$, $1\leq i \leq N$ are independent and identically distributed. Then,
it follows that $\mathbb{E}[\bar{u}_i(\cdot)]=\mathbb{E}[\bar{u}_j(\cdot)]$, for $1\leq i,j \leq N$ and $j\neq i$. Denote $\mu(t)=\mathbb{E}[\bar{u}_i(t)]$, we have
\begin{equation*}
\begin{aligned}
&\int_0^T\mathbb{E}|\frac{1}{N}\overset{N}{\underset{j=1,j\neq i}{\sum }}\bar{u}_{j}(t)-\mu(t)|^2dt\leq\frac{2}{N^2}\int_0^T\mathbb{E}|\mu(t)|^2dt\\
&+\frac{2(N-1)^2}{N^2}\int_0^T\mathbb{E}|\frac{1}{N-1}\overset{N}{\underset{j=1,j\neq i}{\sum }}\bar{u}_{j}(t)-\mu(t)|^2dt\\
&=\frac{2(N-1)}{N^2}\int_0^T\mathbb{E}|\bar{u}_{j}(t)-\mu(t)|^2dt+\frac{2}{N^2}\int_0^T\mathbb{E}|\mu(t)|^2dt\\
&=O(\frac{1}{N}).
\end{aligned}
\end{equation*}
By using the fact that $\mathbb{E}\underset{0\leq t\leq T}{\sup}|y_j(t)|^2\leq K$, and recalling $\mathbb{E}\int_{0}^{T}|u_{i}(t)|^{2}dt\leq K$,  we have
\begin{equation*}
\begin{aligned}
&\frac{K}{N^2}\mathbb{E}\int_0^T|u_i(t)|^2dt+\frac{K}{N^2}\mathbb{E}\overset{N}{\underset{j=1}{\sum}}\int_0^T|C(t)y_j(t)+\beta(t)(y^{(N)}(t)-m(t))
+\beta(t)m(t)+\sigma(t)|^2dt\\
&+\frac{K}{N^2}\mathbb{E}\overset{N}{\underset{j=1}{\sum}}\int_0^T|C_0(t)y_j(t)+\beta_0(t)(y^{(N)}(t)-m(t))+\beta_0(t)m(t)+\sigma_0(t)|^2dt\\
&\leq \frac{K}{N}\Big(1+\mathbb{E}\int_0^T|y^{(N)}(t)-m(t)|^{2}dt\Big).
\end{aligned}
\end{equation*}
Moreover, by i.i.d property of $\bar{u}_i(\cdot)$, we get
$\frac{K}{N^2}\mathbb{E}\overset{N}{\underset{j=1,j\neq i}{\sum }}\int_0^T|\bar{u}_j(t)|^2dt=O(\frac{1}{N})$.
Then, we have
\begin{equation*}
\mathbb{E}\underset{0\leq t\leq T}{\sup }|y^{(N)}(t)-m(t)|^{2}\leq K\mathbb{E
}\int_{0}^{T}|y^{(N)}(t)-m(t)|^{2}dt+O(\frac{1}{N}).
\end{equation*}
From Gronwall's inequality, we obtain \eqref{est 4}.

Secondly, from \eqref{perturbed state yi}, \eqref{bar yi} and \eqref{est 4},
by using Gronwall's inequality, we obtain \eqref{est 5}.
%The proof is completed.
\end{proof}
\smallskip

\begin{lemma}\label{difference of cf 2}
Let \textup{(A1)}-\textup{(A2)} hold, for any $1\leq i\leq N$, we have
\begin{equation*}
\Big|\mathcal{J}_{i}(u_{i}(\cdot ),\bar{u}_{-i}(\cdot ))-J_{i}(u_{i}(\cdot ))\Big|=O(\frac{1}{\sqrt{N}}).
\end{equation*}
\end{lemma}
The proof is similar to Lemma \ref{difference of cf 1}, so we omit it here.
\smallskip

Finally, let us give the proof of Theorem \ref{main theorem}.

\emph{Proof of Theorem \ref{main theorem}}.
%We consider the $\varepsilon$-Nash equilibrium for $\mathcal{A}_i$, $1\leq i \geq N$.
Combining Lemmas \ref{difference of cf 1} and \ref{difference of cf 2}, we have
\begin{equation*}
  \begin{aligned}
 &\mathcal{J}_i(\bar{u}_i(\cdot),\bar{u}_{-i}(\cdot))=J_i(\bar{u}_i(\cdot))+O(\frac{1}{\sqrt{N}})\\
\leq &J_i(u_i(\cdot))+O(\frac{1}{\sqrt{N}})
  =\mathcal{J}_i(u_i(\cdot),\bar{u}_{-i}(\cdot))+O(\frac{1}{\sqrt{N}}).
  \end{aligned}
\end{equation*}
Thus, Theorem \ref{main theorem} holds by taking $\varepsilon=O(\frac{1}{\sqrt{N}})$.\hfill$\blacksquare$

\section{Feedback decentralized strategies}\label{sec3}
In this section, we will further study the above decentralized control strategies \eqref{gcontrol} which can be represented as the feedback of filtered state by Riccati approach as given in the following theorem. Although we have obtained the decentralized control strategies \eqref{gcontrol} for \textbf{Problem (\uppercase\expandafter{\romannumeral2})}, it is not an implementable control policy, since it involves $\mathbb{E}[p_i(\cdot)|\mathcal{F}_{\cdot}^{W_i}]$, $\mathbb{E}[q_i(\cdot)|\mathcal{F}_{\cdot}^{W_i}]$ and $\mathbb{E}[q_{i,0}(\cdot)|\mathcal{F}_{\cdot}^{W_i}]$, where $(p_i(\cdot),q_i(\cdot),q_{i,0}(\cdot))$ is the solution to FBSDE \eqref{H system}. As usual, we would like to obtain a feedback representation of the decentralized control strategy via Riccati equations.

For $1\leq i \leq N$, to simplify symbols, we denote $\hat{f}_i(t)=\mathbb{E}[f_i(t)|\mathcal{F}_t^{W_i}]$ as the filtering of $f_i(t)$ w.r.t. $\mathcal{F}_t^{W_i}$. Then, the decentralized strategies \eqref{gcontrol} can be further expressed by
\begin{equation}\label{open-loop}
\begin{aligned}
\bar{u}_{i}(t)=&R(t)^{-1}\big(B^{\top }(t)\hat{\bar{p}}
_{i}(t)+D^{\top}(t)\hat{\bar{q}}_{i}(t)+D_0^{\top }(t)\hat{\bar{q}}_{i,0}(t)\big).
\end{aligned}
\end{equation}

Then we recall the consistency conditional system \eqref{CC system-1} as
\begin{equation}\label{CC system-4}
\left\{
\begin{aligned}
d\bar{z}_{i}(t)=&\{A(t)\bar{z}_{i}(t)+B(t)R^{-1}(t)(B^{\top }(t)\hat{\bar{p}}_{i}(t)+D^{\top }(t)\hat{\bar{q}}_{i}(t)\\
&+D_0^{\top}(t)\hat{\bar{q}}_{i,0}(t))+\alpha(t) \mathbb{E}[\bar{z}_{i}(t)|\mathcal{F}_t^{W_0}]+b(t)\}dt\\
&+\{C(t)\bar{z}_{i}(t)+D(t)R^{-1}(t)(B^{\top}(t)\hat{\bar{p}}_{i}(t)+D^{\top }(t)\hat{\bar{q}}_{i}(t)\\
&+D_0^{\top}(t)\hat{\bar{q}}_{i,0}(t))+\beta(t) \mathbb{E}[\bar{z}_{i}(t)|\mathcal{F}_t^{W_0}]+\sigma (t)\}dW_{i}(t)\\
&+\{C_0(t)\bar{z}_{i}(t)+D_0(t)R^{-1}(t)(B^{\top }(t)\hat{\bar{p}}_{i}(t)+D^{\top }(t)\hat{\bar{q}}_{i}(t)\\
&+D_0^{\top}(t)\hat{\bar{q}}_{i,0}(t))+\beta_0(t) \mathbb{E}[\bar{z}_{i}(t)|\mathcal{F}_t^{W_0}]+\sigma_0(t)\}dW_{0}(t), \\
d\bar{p}_{i}(t)=&-[A^{\top }(t)\bar{p}_{i}(t)+C^{\top}(t)\bar{q}_{i}\left( t\right)
+C_0^{\top }(t)\bar{q}_{i,0}(t)\\
&-Q(t)(\bar{z}_i\left(
t\right) -\mathbb{E}[\bar{z}_{i}(t)|\mathcal{F}_t^{W_0}])]dt\\
&+\bar{q}_{i}\left( t\right) dW_{i}(t)+\bar{q}_{i,0}\left( t\right) dW_{0}(t), \\
\bar{z}_{i}(0)=&x,\quad \bar{p}_{i}(T)=-G\bar{z}_{i}\left( T\right).
\end{aligned}
\right.
\end{equation}

\begin{theorem}\label{theoorem Riccati}
Let \textup{(A1)}-\textup{(A2)} hold, then the decentralized strategies for \textbf{Problem (\uppercase\expandafter{\romannumeral2})} can be represented as
\begin{equation}
\begin{aligned}\label{scontrol}
\bar{u}_i(t)=
&-\Sigma(t)^{-1}\big(B^{\top}(t)P(t)+D^{\top}(t)P(t)C(t)
+D_0^{\top}(t)P(t)C_0(t)\big)\hat{\bar{z}}_i(t)\\
&-\Sigma(t)^{-1}\big(B^{\top}(t)\Gamma(t)+D^{\top}P(t)\beta(t)+D_0^{\top}(t)P(t)\beta_0(t)\big)\mathbb{E}[m(t)]\\
&-\Sigma(t)^{-1}(B^{\top}(t)\Phi(t)+D^{\top}(t)P(t)\sigma(t)+D_0^{\top}(t)P(t)\sigma_0(t)),\quad 1\leq i \leq N,
\end{aligned}
\end{equation}
with
\begin{equation}\label{tildeR}
\Sigma(t)=R(t)+D^{\top}(t)P(t)D(t)+D_0^{\top}(t)P(t)D_0(t),
\end{equation}
%\begin{equation}\label{tildeP}
%\widetilde{P}(t)=P(t)B(t)+C^{\top}(t)P(t)D(t),
%\end{equation}
where $P(\cdot)$ and $\Gamma(\cdot)$ solve the following Riccati equations, respectively
\begin{equation}\label{P-1}
\left\{
\begin{aligned}
&\dot{P}(t)+P(t)A(t)+A^{\top}(t)P(t)+C^{\top}(t)P(t)C(t)+C_0^{\top}(t)P(t)C_0(t)\\
&+Q(t)
-(P(t)B(t)+C^{\top}(t)P(t)D(t)+C_0^{\top}(t)P(t)D_0(t))\Sigma(t)^{-1}\\
&\times\big(B^{\top}(t)P(t)+D^{\top}(t)P(t)C(t)+D_0^{\top}(t)P(t)C_0(t)\big)=0,\\
&P(T)=G,
\end{aligned}
\right.
\end{equation}
and
\begin{equation}\label{Gamma-1}
\left\{
\begin{aligned}
&\dot{\Gamma}(t)+\Gamma(t)\Big(A(t)-B(t)\Sigma(t)^{-1}\big(B^{\top}(t)P(t)+D^{\top}(t)P(t)C(t)\\
&+D_0^{\top}(t)P(t)C_0(t)\big)\Big)+\big(A(t)-B(t)\Sigma(t)^{-1}\\
&\times(B^{\top}(t)P(t)+D^{\top}(t)P(t)C(t)+D_0^{\top}(t)P(t)C_0(t))\big)^{\top}\Gamma(t)\\
&-\Gamma(t)B(t)\Sigma(t)^{-1}\big(D^{\top}(t)P(t)\beta(t)+D_0^{\top}(t)P(t)\beta_0(t)\big)\\
&+C^{\top}P(t)\beta(t)+C_0^{\top}P(t)\beta_0(t)\\
&-\big(P(t)B(t)+C^{\top}P(t)D(t)+C_0^{\top}P(t)D_0(t)\big)\Sigma(t)^{-1}\\
&\times\big(D^{\top}P(t)\beta(t)+D_0^{\top}P(t)\beta_0(t)\big)\\
&+(P(t)+\Gamma(t))\alpha(t) -\Gamma(t)B(t)\Sigma(t)^{-1}B^{\top}(t)\Gamma(t)-Q(t)=0,\\
&\Gamma(T)=0,
\end{aligned}
\right.
\end{equation}
the deterministic function $\Phi(\cdot)$ solves the following standard ordinary differential equation (ODE)
\begin{equation}
\left\{
\begin{aligned}\label{Phi-11}
&\dot{\Phi}(t)+\Big(A^{\top}(t)-\big(P(t)B(t)+C^{\top}(t)P(t)D(t)+C_0^{\top}(t)P(t)D_0(t)\big)\\
&\qquad\quad\times\Sigma(t)^{-1}B^{\top}(t)-\Gamma(t)B(t)\Sigma(t)^{-1}B^{\top}(t)\Big)\Phi(t)\\
&\quad+\Big(C^{\top}(t)
-\big(P(t)B(t)+C^{\top}(t)P(t)D(t)+C_0^{\top}(t)P(t)D_0(t)\big)\\
&\qquad\quad\times\Sigma(t)^{-1}D^{\top}(t)-\Gamma(t)B(t)\Sigma(t)^{-1}D^{\top}(t)\Big)P(t)\sigma(t)\\
&\quad+\Big(C_0^{\top}(t)-\big(P(t)B(t)+C^{\top}(t)P(t)D(t)+C_0^{\top}(t)P(t)D_0(t)\big)\\
&\qquad\quad\times\Sigma(t)^{-1}D_0^{\top}(t)
-\Gamma(t)B(t)\Sigma(t)^{-1}D_0^{\top}(t)\Big)P(t)\sigma_0(t)\\
&+(P(t)+\Gamma(t))b(t)=0,\\
&\Phi(T)=0,
\end{aligned}
\right.
\end{equation}
the limit value of the state-average $m(\cdot)$ solves the following SDE
\begin{equation}\label{m00}
\left\{
\begin{aligned}
&dm(t)=\big\{(A(t)+\alpha(t))m(t) -B(t)\Sigma(t)^{-1}(B^{\top}(t)P(t)\\
&\quad\quad\quad +D^{\top}(t)P(t)C(t)+D_0^{\top}(t)P(t)C_0(t)+B^{\top}(t)\Gamma(t)\\
&\quad\quad\quad+D^{\top}(t)P(t)\beta(t)+D_0^{\top}(t)P(t)\beta_0(t)
)\mathbb{E}[m(t)]
+b(t)-B(t)\Sigma(t)^{-1}\\
&\quad \quad\quad\times(B^{\top}(t)\Phi(t)+D^{\top}(t)P(t)\sigma(t)+D_0^{\top}(t)P(t)\sigma_0(t))\big\}dt,\\
&\quad\quad\quad+\big\{(C_0(t)+\beta(t))m(t)-D_0(t)\Sigma(t)^{-1}(B^{\top}(t)P(t)+D^{\top}(t)P(t)C(t)\\
&\quad\quad\quad+D_0^{\top}(t)P(t)C_0(t)+B^{\top}(t)\Gamma(t)+D^{\top}(t)P(t)\beta(t)\\
&\quad\quad\quad+D_0^{\top}(t)P(t)\beta_0(t))\mathbb{E}m(t)+\sigma_0(t)
-D_0(t)\Sigma(t)^{-1}\\
&\quad\quad\quad\times(B^{\top}(t)\Phi(t)+D^{\top}(t)P(t)\sigma(t)+D_0^{\top}(t)P(t)\sigma_0(t))\big\}dW_0(t),\\
&m(0)=x,
\end{aligned}
\right.
\end{equation}
and the optimal filtering $\hat{\bar{z}}_i(\cdot)$ solves the following SDE
\begin{equation}
\left\{
\begin{aligned}\label{hat z}
d\hat{\bar{z}}_i(t)=&\Big\{\Big(A(t)-B(t)\Sigma(t)^{-1}\big(B^{\top}(t)P(t)+D^{\top}(t)P(t)C(t)\\
&+D_0^{\top}(t)P(t)C_0(t)\big)\Big)\hat{\bar{z}}_i(t)
+\Big(\alpha(t)-B(t)\Sigma(t)^{-1}\big(B^{\top}(t)\Gamma(t)\\
&+D^{\top}P(t)\beta(t)
+D_0^{\top}(t)P(t)\beta_0(t)\big)
\Big)\mathbb{E}[m(t)]-B(t)\Sigma(t)^{-1}(B^{\top}(t)\Phi(t)\\
&+D^{\top}(t)P(t)\sigma(t)+D_0^{\top}(t)P(t)\sigma_0(t))+b(t)\Big\}dt\\
&+\Big\{\Big(C(t)-D(t)\Sigma(t)^{-1}\big(B^{\top}(t)P(t)+D^{\top}(t)P(t)C(t)\\
&+D_0^{\top}(t)P(t)C_0(t)\big)\Big)\hat{\bar{z}}_i(t)
+\Big(\beta(t)-D(t)\Sigma(t)^{-1}\big(B^{\top}(t)\Gamma(t)\\
&+D^{\top}P(t)\beta(t)
+D_0^{\top}(t)P(t)\beta_0(t)\big)
\Big)\mathbb{E}[m(t)]\\
&-D(t)\Sigma(t)^{-1}(B^{\top}(t)\Phi(t)+D^{\top}(t)P(t)\sigma(t)+D_0^{\top}(t)P(t)\sigma_0(t))\\
&+\sigma(t)\Big\}dW_i(t),\\
&\hat{\bar{z}}_i(0)=x.
\end{aligned}
\right.
\end{equation}
\end{theorem}

\begin{proof}
Due to the terminal condition and structure of \eqref{CC system-4}, we suppose
\begin{equation}\label{pdecompose-0}
\bar{p}_i(t)=-P(t)\bar{z}_i(t)-\Gamma(t)\mathbb{E}[\bar{z}_i(t)]-\Phi(t), \quad 1\leq i \leq N,
\end{equation}
with $P(T)=G$, $\Gamma(T)=0$ and $\Phi(T)=0$, where $P(\cdot)$, $\Gamma(\cdot)$ and $\Phi(\cdot)$ will be specified later.
Applying It\^o's formula to \eqref{pdecompose-0}, we have
\begin{equation}\label{barp-0}
\begin{aligned}
d&\bar{p}_i(t)=\Big[-\big(\dot{P}(t)+P(t)A(t)\big)\bar{z}_i(t)
-\big(\dot{\Gamma}(t)+\Gamma(A(t)+\alpha(t))\big)\mathbb{E}[\bar{z}_i(t)]\\
&-P(t)\alpha(t)\mathbb{E}[\bar{z}_i(t)|\mathcal{F}_t^{W_0}]
-P(t)B(t)\bar{u}_i(t)-\Gamma(t)B(t)\mathbb{E}[\bar{u}_i(t)]\\
&-\dot{\Phi}(t)-(P(t)+\Gamma(t))b(t)\Big]dt\\
&-P(t)[C(t)\bar{z}_{i}(t) +D(t)\bar{u}_i(t)+\beta(t)\mathbb{E}[\bar{z}_i(t)|\mathcal{F}_t^{W_0}]+\sigma(t)]dW_i(t)\\
&-P(t)[C_0(t)\bar{z}_{i}(t) +D_0(t)\bar{u}_i(t)+\beta_0(t)\mathbb{E}[\bar{z}_i(t)|\mathcal{F}_t^{W_0}]+\sigma_0(t)]dW_0(t).
\end{aligned}
\end{equation}
Comparing with the diffusion terms in the second equation of \eqref{CC system-4}, we get
\begin{equation}\label{krelation-0}
\begin{aligned}
\bar{q}_i(t)=&-P(t)[C(t)\bar{z}_{i}(t) +D(t)\bar{u}_i(t)+\beta(t)\mathbb{E}[\bar{z}_i(t)|\mathcal{F}_t^{W_0}]+\sigma (t)],\\
\bar{q}_{i,0}(t)=&-P(t)[C_0(t)\bar{z}_{i}(t) +D_0(t)\bar{u}_i(t)+\beta_0(t)\mathbb{E}[\bar{z}_i(t)|\mathcal{F}_t^{W_0}]+\sigma_0 (t)].
\end{aligned}
\end{equation}
By taking conditional expectation w.r.t. $\mathcal{F}_t^{W_i}$ of \eqref{pdecompose-0} and \eqref{krelation-0}, we have
\begin{equation*}
\begin{aligned}
\hat{\bar{p}}_i(t)=&-P(t)\hat{\bar{z}}_i(t)-\Gamma(t)\mathbb{E}[\bar{z}_i(t)]-\Phi(t),\\
\hat{\bar{q}}_i(t)=&-P(t)[C(t)\hat{\bar{z}}_{i}(t) +D(t)\bar{u}_i(t)+\beta\mathbb{E}[\bar{z}_i(t)]+\sigma(t)],\\
\hat{\bar{q}}_{i,0}(t)=&-P(t)[C_0(t)\hat{\bar{z}}_{i}(t) +D_0(t)\bar{u}_i(t)+\beta_0\mathbb{E}[\bar{z}_i(t)]+\sigma_0(t)].
\end{aligned}
\end{equation*}
Then substituting them into \eqref{open-loop}, we can derive
\begin{equation}\label{u-0}
\begin{aligned}
\bar{u}_i(t)=
&-\Sigma(t)^{-1}\big(B^{\top}(t)P(t)+D^{\top}(t)P(t)C(t)+D_0^{\top}(t)P(t)C_0(t)\big)\hat{\bar{z}}_i(t)\\
&-\Sigma(t)^{-1}\big(B^{\top}(t)\Gamma(t)+D^{\top}P(t)\beta(t)+D_0^{\top}(t)P(t)\beta_0(t)\big)\mathbb{E}[\bar{z}_i(t)]\\
&-\Sigma(t)^{-1}(B^{\top}(t)\Phi(t)+D^{\top}(t)P(t)\sigma(t)+D_0^{\top}(t)P(t)\sigma_0(t)),
\\
&\quad 1\leq i \leq N.
\end{aligned}
\end{equation}
By recalling that $m(\cdot)=\mathbb{E}[\bar{z}_i(\cdot)|\mathcal{F}_t^{W_0}]$, we have $\mathbb{E}[m(\cdot)]=\mathbb{E}[\bar{z}_i(\cdot)]$, then \eqref{scontrol} holds.
Moreover, we have
\begin{equation}\label{Eu-0}
\begin{aligned}
\mathbb{E}[\bar{u}_i(t)]=&-\Sigma(t)^{-1}\big(B^{\top}(t)P(t)+D^{\top}(t)P(t)C(t)+D_0^{\top}(t)P(t)C_0(t)\\
&+B^{\top}(t)\Gamma(t)+D^{\top}(t)P(t)\beta(t)+D_0^{\top}(t)P(t)\beta_0(t)\big)\mathbb{E}[\bar{z}_i(t)]\\
&-\Sigma(t)^{-1}\big(B^{\top}(t)\Phi(t)+D^{\top}(t)P(t)\sigma(t)+D_0^{\top}(t)P(t)\sigma_0(t)\big), \\&\quad1\leq i \leq N.
\end{aligned}
\end{equation}
Now, let us consider the equations for $P(\cdot)$, $\Gamma(\cdot)$ and $\Phi(\cdot)$.
Comparing the drift terms of \eqref{barp-0} with second equation in \eqref{CC system-4}, by noticing \eqref{pdecompose-0} and \eqref{krelation-0}, we obtain
\begin{equation}\label{drift-0}
\begin{aligned}
&\Big(\dot{P}(t)+P(t)A(t)+A^{\top}(t)P(t)+C^{\top}(t)P(t)C(t)+C_0^{\top}(t)P(t)C_0(t)\\
&+Q(t)\Big)\bar{z}_i(t)
+\Big(P(t)B(t)+C^{\top}(t)P(t)D(t)+C_0^{\top}(t)P(t)D_0(t)\Big)\bar{u}_i(t)\\
&+\Big(\dot{\Gamma}(t)+\Gamma(t)(A(t)+\alpha(t))+A^{\top}(t)\Gamma(t)\Big)\mathbb{E}[\bar{z}_i(t)]
+\Gamma(t)B(t)\mathbb{E}[\bar{u}_i(t)]\\
&
+\Big(P(t)\alpha(t)+C^{\top}(t)P(t)D(t)+C_0^{\top}(t)P(t)D_0(t)-Q(t)\Big)
\\
&\times\mathbb{E}[\bar{z}_i(t)|\mathcal{F}_t^{W_0}]
+\dot{\Phi}(t)+P(t)b(t)+\Gamma(t)b(t)+A^{\top}(t)\Phi(t)\\
&+C^{\top}(t)P(t)\sigma(t)+C_0^{\top}(t)P(t)\sigma_0(t)=0.
\end{aligned}
\end{equation}
By substituting \eqref{u-0} and \eqref{Eu-0} into \eqref{drift-0} and taking conditional expectation w.r.t. $\mathcal{F}_t^{W_i}$ of \eqref{drift-0},
we obtain the equations for $P(\cdot)$, $\Gamma(\cdot)$ and $\Phi(\cdot)$ which solve \eqref{P-1}, \eqref{Gamma-1} and \eqref{Phi-11}, respectively.

In addition, by substituting \eqref{Eu-0} into \eqref{m-1}, it is easy to show that $m(\cdot)$ solves \eqref{m00}. Moreover, from \eqref{s-lstate} and \eqref{scontrol}, then equation \eqref{hat z} of the optimal filtering $\hat{\bar{z}}_i(\cdot)$ is obtained, which completes the proof.
\end{proof}

To summarize, we obtain that $(P(\cdot),\Gamma(\cdot),\Phi(\cdot))$ solves the following Riccati system
\begin{equation}
\left\{
\begin{aligned}\label{RCC}
&\dot{P}(t)+P(t)A(t)+A^{\top}(t)P(t)+C^{\top}(t)P(t)C(t)+C_0^{\top}(t)P(t)C_0(t)\\
&+Q(t)
-(P(t)B(t)+C^{\top}(t)P(t)D(t)+C_0^{\top}(t)P(t)D_0(t))\Sigma(t)^{-1}\\
&\times\big(B^{\top}(t)P(t)+D^{\top}(t)P(t)C(t)+D_0^{\top}(t)P(t)C_0(t)\big)=0,\\
&\dot{\Gamma}(t)+\Gamma(t)\Big(A(t)-B(t)\Sigma(t)^{-1}\big(B^{\top}(t)P(t)+D^{\top}(t)P(t)C(t)\\
&+D_0^{\top}(t)P(t)C_0(t)\big)\Big)+\big(A(t)-B(t)\Sigma(t)^{-1}\\
&\times(B^{\top}(t)P(t)+D^{\top}(t)P(t)C(t)+D_0^{\top}(t)P(t)C_0(t))\big)^{\top}\Gamma(t)\\
&-\Gamma(t)B(t)\Sigma(t)^{-1}\big(D^{\top}(t)P(t)\beta(t)+D_0^{\top}(t)P(t)\beta_0(t)\big)\\
&+C^{\top}P(t)\beta(t)+C_0^{\top}P(t)\beta_0(t)\\
&-\big(P(t)B(t)+C^{\top}P(t)D(t)+C_0^{\top}P(t)D_0(t)\big)\Sigma(t)^{-1}\\
&\times\big(D^{\top}P(t)\beta(t)+D_0^{\top}P(t)\beta_0(t)\big)\\
&+(P(t)+\Gamma(t))\alpha(t) -\Gamma(t)B(t)\Sigma(t)^{-1}B^{\top}(t)\Gamma(t)-Q(t)=0,\\
&\dot{\Phi}(t)+\Big(A^{\top}(t)-\big(P(t)B(t)+C^{\top}(t)P(t)D(t)+C_0^{\top}(t)P(t)D_0(t)\big)\\
&\qquad\quad\times\Sigma(t)^{-1}B^{\top}(t)-\Gamma(t)B(t)\Sigma(t)^{-1}B^{\top}(t)\Big)\Phi(t)\\
&\quad+\Big(C^{\top}(t)
-\big(P(t)B(t)+C^{\top}(t)P(t)D(t)+C_0^{\top}(t)P(t)D_0(t)\big)\\
&\qquad\quad\times\Sigma(t)^{-1}D^{\top}(t)-\Gamma(t)B(t)\Sigma(t)^{-1}D^{\top}(t)\Big)P(t)\sigma(t)\\
&\quad+\Big(C_0^{\top}(t)-\big(P(t)B(t)+C^{\top}(t)P(t)D(t)+C_0^{\top}(t)P(t)D_0(t)\big)\\
&\qquad\quad\times\Sigma(t)^{-1}D_0^{\top}(t)
-\Gamma(t)B(t)\Sigma(t)^{-1}D_0^{\top}(t)\Big)P(t)\sigma_0(t)\\
&+(P(t)+\Gamma(t))b(t)=0,\\
&P(T)=G,\ \Gamma(T)=0,\ \Phi(T)=0,
\end{aligned}
\right.
\end{equation}
where $\Sigma(t)=R(t)+D^{\top}(t)P(t)D(t)+D_0^{\top}(t)P(t)D_0(t)$.

Noticing that due to the term $P(t)B(t)+C^{\top}(t)P(t)D(t)+C_0^{\top}(t)P(t)D_0(t)$ in \eqref{P-1} and
\begin{equation*}
D^{\top}(t)P(t)D(t)+D_0^{\top}(t)P(t)D_0(t)  \neq(D^{\top}(t)+D_0^{\top}(t))^{\top}P(t)(D^{\top}(t)+D_0^{\top}(t)),
\end{equation*}
it means that \eqref{P-1} is not standard Riccati equation.
Concerning \eqref{Gamma-1}, it is more complicated since the non-symmetry comes from not only the term
$\Gamma(t)B(t)\Sigma(t)^{-1}\big(D^{\top}(t)P(t)\beta(t)+D_0^{\top}(t)P(t)\beta_0(t)\big)$
but also the term $C^{\top}P(t)\beta(t)+C_0^{\top}P(t)\beta_0(t)-\big(P(t)B(t)+C^{\top}P(t)D(t)+C_0^{\top}(t)P(t)D_0(t)\big)\Sigma(t)^{-1}\big(D^{\top}P(t)\beta(t)+D_0^{\top}P(t)\beta_0(t)\big)$
and the term $\Gamma(t)\alpha(t)$. Concerning the equation of $\Phi(\cdot)$, it is a ODE once $P(\cdot)$ and $\Gamma(\cdot)$ are determined. In general, we see that the solvability of $(P(\cdot),\Gamma(\cdot),\Phi(\cdot))$ is very complicated.

Firstly, we investigate the well-posedness of equation \eqref{P-1}.

\begin{lemma}\label{Plemma}
Let \textup{(A1)}-\textup{(A2)} hold. Then the Riccati equation \eqref{P-1} admits a unique solution $P(\cdot)\in C([0,T];\mathcal{S}_+^n)$.
\end{lemma}

\begin{proof}
Firstly, we prove \eqref{P-1} admits at most one solution $P(\cdot)\in C([0,T];\mathcal{S}_+^n)$. Suppose $P_1(\cdot)$ and $P_2(\cdot)$ are two solutions of \eqref{P-1} and we set $\widehat{P}(\cdot)=P_1(\cdot)-P_2(\cdot)$, then $\widehat{P}(\cdot)$ satisfies
\begin{equation*}
\left\{
\begin{aligned}
&\dot{\widehat{P}}(t)+\widehat{P}(t)A(t)+A^{\top}(t)\widehat{P}(t)+C^{\top}(t)\widehat{P}(t)C(t)\\
&+C_0^{\top}(t)\widehat{P}(t)C_0(t)
-(\widehat{P}(t)B(t)+C^{\top}(t)\widehat{P}(t)D(t)\\
&+C_0^{\top}(t)\widehat{P}(t)D_0(t))\Sigma_1(t)^{-1}\\
&\times(P_1(t)B(t)+C^{\top}(t)P_1(t)D(t)+C_0^{\top}(t)P_1(t)D_0(t))^{\top}\\
&-(P_2(t)B(t)+C^{\top}(t)P_2(t)D(t)+C_0^{\top}(t)P_2(t)D_0(t))\\
&\times\Sigma_2(t)^{-1}(\widehat{P}(t)B(t)+C^{\top}(t)\widehat{P}(t)D(t)+C_0^{\top}(t)\widehat{P}(t)D_0(t))^\top\\
&+(P_2(t)B(t)+C^{\top}(t)P_2(t)D(t)+C_0^{\top}(t)P_2(t)D_0(t))\\
&\times\Sigma_1(t)^{-1}(D^{\top}(t)\widehat{P}(t)D(t)
+D_0^{\top}(t)\widehat{P}(t)D_0(t))\Sigma_2(t)^{-1}\\
&\times(P_1(t)B(t)+C^{\top}(t)P_1(t)D(t)+C_0^{\top}(t)P_1(t)D_0(t))^{\top}=0,\\
&\widehat{P}(T)=0,
\end{aligned}
\right.
\end{equation*}
where $\Sigma_i(t)=R(t)+D^{\top}(t)P_i(t)D(t)+D_0^{\top}(t)P_i(t)D_0(t)$, $i=1,2$.
Noting that for $i=1,2$, $\Sigma_i(t)^{-1}\leq R(t)^{-1}$ and $|\Sigma_i(t)^{-1}|<\infty$,
recalling $R\gg0$ (which implies $R^{-1}\in L^{\infty }(0,T;\mathcal{S}^{k})$). Consequently, $|\Sigma_1(t)^{-1}|$ and $|\Sigma_2(t)^{-1}|$ are uniformly bounded. Using Gronwall's inequality, we get $\widehat{P}(t)=0$. This means that the uniqueness of $P(\cdot)$ is proved.

Secondly, let us focus on the existence of solution to \eqref{P-1}. Motivated by \cite{YZ1999},
we can verify that equation \eqref{P-1} is equivalent to the following equation
\begin{equation}
\left\{
\begin{aligned}\label{P2}
&\dot{P}(t)+P(t)\widehat{A}(t)+\widehat{A}^\top(t)P(t)+\widehat{C}^\top(t)P(t)\widehat{C}(t)\\
&+\widehat{C}_0^\top(t)P(t)\widehat{C}_0(t)+\widehat{Q}(t)=0,\\
&P(T)=G,
\end{aligned}
\right.
\end{equation}
where
\begin{equation}\label{hat ACQ}
\left\{
\begin{aligned}
&\widehat{A}(t)=A(t)-B(t)\Psi(t),\quad \widehat{C}(t)=C(t)-D(t)\Psi(t),\\
&\widehat{C}_0(t)=C_0(t)-D_0(t)\Psi(t),\\
&\widehat{Q}(t)=Q(t)+\Psi^{\top}(t)R(t)\Psi(t),\\
&\Psi(t)=(R(t)+D^{\top}(t)P(t)D(t)+D_0^{\top}(t)P(t)D_0(t))^{-1}\\
&\quad \quad \quad \times (P(t)B(t)+C^\top(t)P(t)D(t)+C_0^\top(t)P(t)D_0(t))^\top.\\
\end{aligned}
\right.
\end{equation}
We mention that different from \cite{YZ1999}, we have additional terms in \eqref{tildeR} and \eqref{P-1}, thus we need modify  $\Psi(t)$, which includes $D_0^{\top}(t)P(t)D_0(t)$.

%This is a quite challenging issue, we emphasise that it is not easy to guess \eqref{hat ACQ}, especially since $D_0^{\top}(t)P(t)D_0(t)$ depends on $P$ itself, but it is usually to guess that $\widehat{Q}(t)$ should be $Q(t)+\Psi^{\top}(t)R(t)\Psi(t)$. In this manner, the following iterative method is invalid.

Next, with the help of modified iterative method and mathematical induction, we can prove the existence of solution to \eqref{P2}, which yields that \eqref{P-1} also has a solution. To do this, we set
\begin{equation}\label{P0}
\left\{
\begin{aligned}
&\dot{P}_0(t)+P_0(t)A(t)+A^\top(t)P_0(t)+C^\top(t)P_0(t)C(t)\\
&\qquad+C_0^\top(t)P_0(t)C_0(t)+Q(t)=0,\\
&P_0(T)=G,
\end{aligned}
\right.
\end{equation}
which admits a unique $P_0(\cdot)\in C([0,T];\mathcal{S}_+^n)$ by noticing $Q(\cdot)\geq 0$ (see Lemma 7.3 of \cite{YZ1999}).
For $i\ge0$, we define
\begin{equation}\label{iterative}
\left\{
\begin{aligned}
&\Psi_i(t)=(R(t)+D^{\top}(t)P_i(t)D(t)+D_0^{\top}(t)P_i(t)D_0(t))^{-1}\\
&\quad\times (P_i(t)B(t)+C^\top(t)P_i(t)D(t)+C_0^\top(t)P_i(t)D_0(t))^\top,\\
&\widehat{A}_i(t)=A(t)-B(t)\Psi_i(t),\quad \widehat{C}_i(t)=C(t)-D(t)\Psi_i(t),\\
&\widehat{C}_{0,i}(t)=C_0(t)-D_0(t)\Psi_i(t), \\ &\widehat{Q}_i(t)=Q(t)+\Psi_i^{\top}(t)R(t)\Psi_i(t).\\
\end{aligned}
\right.
\end{equation}
Noticing that \eqref{iterative} is also quite different from \cite{YZ1999} due to the appearance of $D_0^{\top}(t)P_i(t)D_0(t)$.

Let $P_{i+1}(t)$ be the solution of
\begin{equation}\label{Pi+1}
\left\{
\begin{aligned}
&\dot{P}_{i+1}(t)+P_{i+1}(t)\widehat{A}_i(t)+\widehat{A}_i^{\top}(t)P_{i+1}(t)
+\widehat{Q}_i(t)\\
&+\widehat{C}_i^\top(t)P_{i+1}(t)\widehat{C}_i(t)
+\widehat{C}_{0,i}^\top(t)P_{i+1}(t)\widehat{C}_{0,i}(t)=0,\\
&P_{i+1}(T)=G.
\end{aligned}
\right.
\end{equation}
Noticing that $R\gg0$, $Q\ge0$ and $G\geq0$, by using Lemma 7.3 of \cite{YZ1999} and mathematical induction, one can check that for $i\ge 0$, $P_i(\cdot)$ is well defined and $P_i(\cdot)\in C([0,T];\mathcal{S}_+^n)$.
Now we will show that $P_i(\cdot)$, for $i\ge0$, is a decreasing sequence in $C([0,T];\mathcal{S}_+^n)$. For simplicity, we denote $\Delta_{i-1}(t)=P_{i-1}(t)-P_{i}(t)$, $\Upsilon_i(t)=\Psi_{i-1}(t)-\Psi_i(t)$ and $\Upsilon_0(t)=-\Psi_0(t)$.
From \eqref{P0} and \eqref{Pi+1}, when $i=0$, we have
\begin{equation*}
\left\{
\begin{aligned}
&\dot{\Delta}_{0}(t)+P_{0}(t)(A(t)-\widehat{A}_0(t))+\Delta_0(t)\widehat{A}_0(t)+\widehat{A}_0^\top(t)\Delta_0(t)\\
&+(A(t)-\widehat{A}_0(t))^\top P_0(t)+Q(t)-\widehat{Q}_0(t)\\
&+C^\top(t)P_0(t)C(t)-\widehat{C}_0^\top(t)P_0(t)\widehat{C}_0(t)+
\widehat{C}_0^\top(t)\Delta_0(t)\widehat{C}_0(t)\\
&+C_0^\top(t)P_0(t)C_0(t)-\widehat{C}_{0,0}^\top(t)P_0(t)\widehat{C}_{0,0}(t)+
\widehat{C}_{0,0}^\top(t)\Delta_0(t)\widehat{C}_{0,0}(t)=0,\\
&\Delta_{0}(T)=0,
\end{aligned}
\right.
\end{equation*}
which yields that
\begin{equation*}
\begin{aligned}
&-[\dot{\Delta}_0(t)+\Delta_0(t)\widehat{A}_0(t)+\widehat{A}_0^\top(t)\Delta_0(t)+\widehat{C}_0^\top(t)\Delta_0(t)\widehat{C}_0(t)+\widehat{C}_{0,0}^\top(t)\Delta_0(t)\widehat{C}_{0,0}(t)]\\
%=&P_0(t)(A(t)-\widehat{A}_0(t))+(A(t)-\widehat{A}_0(t))^\top P_0(t)+C^\top(t)P_0(t)C(t)-\widehat{C}_0^\top(t)P_0(t)\widehat{C}_0(t)+Q(t)-\widehat{Q}_0(t)\\
%=&\Upsilon^\top_0(t)(R(t)+D^\top(t)P_0(t)D(t)+D_0^\top(t)P_0(t)D_0(t))\Upsilon_0(t)\\
%&-[P_0(t)B(t)+\widehat{C}_0^\top(t)P_0(t)D (t)+\Upsilon^\top_0(t)(R(t)+D_0^\top(t)P_0(t)D_0(t))]\Upsilon_0(t)\\
%&-\Upsilon^\top_0(t)[B^\top(t)P_0(t)+D^\top(t)P_0(t)\widehat{C}_0(t)+(R(t)+D_0^\top(t)P_0(t)D_0(t))\Upsilon_0(t)]\\
=&\Upsilon^\top_0(t)(R(t)+D^\top(t)P_0(t)D(t)+D_0^\top(t)P_0(t)D_0(t))\Upsilon_0(t)\geq 0.
\end{aligned}
\end{equation*}
By noting  $\Delta_0(T)=0$ and Lemma 7.3 of \cite{YZ1999}, we get $P_0(t)\geq P_1(t)$, for all $t\in [0,T]$. For $i\ge1$, suppose $P_{i-1}(t)\geq P_i(t)$, $t\in[0,T]$, it is sufficient to prove $P_{i}(t)\geq P_{i+1}(t)$, $t\in[0,T]$.
Using \eqref{Pi+1} and $\Delta_i(t)=P_i(t)-P_{i+1}(t)$, we have
\begin{equation}\label{Delta2}
\left\{
\begin{aligned}
&\dot{\Delta}_i(t)+P_i(t)(\widehat{A}_{i-1}(t)-\widehat{A}_i(t))+\Delta_i(t)\widehat{A}_i(t)\\
&+(\widehat{A}_{i-1}(t)-\widehat{A}_i(t))^\top P_i(t)+\widehat{A}_i^\top(t)\Delta_i(t)\\
&+\widehat{C}_i^\top(t)\Delta_i(t)\widehat{C}_i(t)+\widehat{C}_{i-1}^\top(t)P_i(t)\widehat{C}_{i-1}(t)-\widehat{C}_i^\top(t)P_i(t)\widehat{C}_i(t)\\
&+\widehat{C}_{0,i}^\top(t)\Delta_i(t)\widehat{C}_{0,i}(t)+\widehat{C}_{0,i-1}^\top(t)P_i(t)\widehat{C}_{0,i-1}(t)-\widehat{C}_{0,i}^\top(t)P_{0,i}(t)\widehat{C}_{0,i}(t)\\
&+\widehat{Q}_{i-1}(t)-\widehat{Q}_i(t)=0,\\
&\Delta_{i}(T)=0.
\end{aligned}
\right.
\end{equation}
According to \eqref{iterative}, we have
\begin{equation*}
\begin{aligned}
&\widehat{A}_{i-1}(t)-\widehat{A}_i(t)=-B(t)\Upsilon_i(t),\quad \widehat{C}_{i-1}(t)-\widehat{C}_{i}(t)=-D(t)\Upsilon_i(t),\quad \widehat{C}_{0,i-1}(t)-\widehat{C}_{0,i}(t)=-D_0(t)\Upsilon_i(t),\\
&\widehat{C}_{i-1}^\top(t)P_i(t)\widehat{C}_{i-1}(t)-\widehat{C}_i^\top(t)P_i(t)\widehat{C}_i(t)\\
=&\Upsilon_i^\top(t)D^\top(t)P_i(t)D(t)\Upsilon_i(t)-\widehat{C}_i^\top(t)P_i(t)D(t)\Upsilon_i(t)-\Upsilon_i^\top(t)D^\top(t)P_i(t)\widehat{C}_i(t),\\
&\widehat{C}_{0,i-1}^\top(t)P_i(t)\widehat{C}_{0,i-1}(t)-\widehat{C}_{0,i}^\top(t)P_i(t)\widehat{C}_{0,i}(t)\\
=&\Upsilon_i^\top(t)D_0^\top(t)P_i(t)D_0(t)\Upsilon_i(t)-\widehat{C}_{0,i}^\top(t)P_i(t)D_0(t)\Upsilon_i(t)-\Upsilon_i^\top(t)D_0^\top(t)P_i(t)\widehat{C}_{0,i}(t),\\
&\widehat{Q}_{i-1}(t)-\widehat{Q}_i(t)\\
=&\Upsilon_i^\top(t)R(t)\Upsilon_i(t)
+\Psi_i^\top(t)R(t)\Upsilon_i(t)
+\Upsilon_i^\top(t)R(t)\Psi_i(t).
\end{aligned}
\end{equation*}
From \eqref{Delta2}, we obtain
\begin{equation}
\begin{aligned}\label{di}
&-[\dot{\Delta}_i(t)+\Delta_i(t)\widehat{A}_i(t)+\widehat{A}_i^\top(t)\Delta_i(t)+\widehat{C}_i^\top(t)\Delta_i(t)\widehat{C}_i(t)+\widehat{C}_{0,i}^\top(t)\Delta_i(t)\widehat{C}_{0,i}(t)]\\
%=&-P_i(t)B(t)\Upsilon_i(t)-\Upsilon_i^\top(t)B^\top(t)P_i(t)+\Upsilon_i^\top(t)D^\top(t)P_i(t)D(t)\Upsilon_i(t)\\
%&-\widehat{C}_i^\top(t)P_i(t)D(t)\Upsilon_i(t)-\Upsilon_i^\top(t)D^\top(t)P_i(t)\widehat{C}_i(t)\\
%&+\Upsilon_i^\top(t)(R(t)+D_0^\top(t)P_i(t)D_0(t))\Upsilon_i(t)+\Psi_i^\top(t)(R(t)+D_0^\top(t)P_i(t)D_0(t))\Upsilon_i(t)\\
%&+\Upsilon_i^\top(t)(R(t)+D_0^\top(t)P_i(t)D_0(t))\Psi_i(t)+\Psi_{i-1}^\top(t)D_0^\top(t)(P_{i-1}(t)-P_i(t))D_0(t)\Psi_{i-1}(t)\\
=&\Upsilon_i^\top(t)(R(t)+D^\top(t)P_i(t)D(t)+D_0^\top(t)P_i(t)D_0(t))\Upsilon_i(t)
%&+\Psi_{i-1}^\top(t)D_0^\top(t)(P_{i-1}(t)-P_i(t))D_0(t)\Psi_{i-1}(t)
\geq 0.
\end{aligned}
\end{equation}
Using $\Delta_i(T)=0$ and Lemma 7.3 of \cite{YZ1999} again, we have  $P_{i}(t)\geq P_{i+1}(t)$, $t\in[0,T]$. Therefore, $\{P_i(\cdot)\}$ is a decreasing sequence in $C([0,T];\mathcal{S}_+^n)$, and thus has a limit denoted by $P(\cdot)$. Now by integrating equation \eqref{Pi+1} on the interval $[t,T]$ and then sending $i$ to infinite, with the help of dominate convergence theorem, one can show that $P(\cdot)$ solves \eqref{P2} (and hence \eqref{P-1}).
The proof is complete.
\end{proof}

\begin{remark}
One may applying the method of Bellman's Pinciple of quasi linearization and monotone convergence result of symmetric matrices as in \cite{W1968} (see also \cite{P1992,T2003}) to prove the well-posdeness of Riccati equation \eqref{P-1}. Here, we use pure algebraic technique which is different to \cite{W1968} to obtain the well-posedness.
\end{remark}

Next, we give the well-posedness of Riccati system \eqref{RCC}.
From Lemma \ref{Plemma}, we know that equation \eqref{P-1} admits a unique solution $P(\cdot)\in C([0,T];\mathcal{S}_+^n)$, once $\Gamma(\cdot)$ is uniquely solved, the well-posedness of $\Phi(\cdot)$ holds by noting that the equation for $\Phi(\cdot)$ is just an ODE.
\begin{theorem}\label{RCCtheorem}
If $\Gamma(\cdot)$ is uniquely solved, then system \eqref{RCC} admits a unique solution $(P(\cdot),\Gamma(\cdot),
\Phi(\cdot))\in $ $C([0,T];\mathcal{S}_+^n\times \mathcal{S}^n\times \mathbb{R}^n)$.
\end{theorem}

\begin{remark}\label{RCCtheorem}
Under certain conditions, $\Gamma(\cdot)$ is uniquely solved. One trivial example is that $n = k = 1$, in which case equation \eqref{Gamma-1} becomes a one-dimensional ODE whose well-posedness is obvious.
 \end{remark}

Now, let us consider some non-trivial case such that \eqref{Gamma-1}  has a unique solution. Firstly, to guarantee the  symmetry of Riccati equation \eqref{Gamma-1} (see the discussion before Lemma \ref{Plemma}), one natural assumption is that
$\alpha(\cdot)=\delta I$ and $\beta(\cdot)=\beta_0(\cdot)=0$, where $\delta$ is a constant. We will keep this assumption in the following arguments.

The following theorem gives a general condition to guarantee the well-posedness of Riccati equation \eqref{Gamma-1}.
\begin{theorem}\label{RCCtheorem Gamma}
Suppose that  $\alpha(\cdot)=\delta I$, $\beta(\cdot)=\beta_0(\cdot)=0$, if
\begin{equation}\label{condition for Pi}
-C^{\top}PD\Sigma^{-1}D_0^{\top}PC_0-C_0^{\top}PD_0\Sigma^{-1}D^{\top}PC\ge0,
\end{equation}
then  Riccati equation \eqref{Gamma-1} admits a unique solution.
\end{theorem}
\begin{proof}
Inspired by Yong \cite{Y2013}, we transform the solvability of $\Gamma(\cdot)$ to the solvability of another Riccati equation $\Pi(\cdot)$. To do this, we set $\Pi(\cdot):=P(\cdot)+\Gamma(\cdot)$, then from  system  \eqref{RCC}, the function $\Pi(\cdot)$ solves the following equation (with $t$ being suppressed)
\begin{equation}
\left\{
\begin{aligned}\label{Pi}
&\dot{\Pi}+\Pi[A-B\Sigma^{-1}(D^{\top}PC+D_0^{\top}PC_0)]\\
&\qquad+[A-B\Sigma^{-1}(D^{\top}PC+D_0^{\top}PC_0)]^{\top}\Pi+\delta\Pi\\
&\qquad+C^{\top}(P
-PD\Sigma^{-1}D^{\top}P)C+C_0^{\top}(P
-PD_0\Sigma^{-1}D_0^{\top}P)C_0\\
&\qquad-C^{\top}PD\Sigma^{-1}D_0^{\top}PC_0-C_0^{\top}PD_0\Sigma^{-1}D^{\top}PC\\
&\qquad-\Pi B\Sigma^{-1}B^{\top}\Pi=0,\\
&\Pi(T)=G.
\end{aligned}
\right.
\end{equation}
Recalling that
$\Sigma=R+D^{\top}PD+D_0^{\top}PD_0$,
we have
\begin{equation*}
\begin{aligned}
&P-PD\Sigma^{-1}D^{\top}P
=P-PD(R+D^{\top}PD+D_0^{\top}PD_0)^{-1}D^{\top}P\\
=&P^{\frac{1}{2}}\big[I-P^{\frac{1}{2}}DR^{-\frac{1}{2}}[I+R^{-\frac{1}{2}}(D^{\top}P^{\frac{1}{2}}P^{\frac{1}{2}}D
+D_0^{\top}P^{\frac{1}{2}}P^{\frac{1}{2}}D_0)R^{-\frac{1}{2}}]^{-1} R^{-\frac{1}{2}}D^{\top}P^{\frac{1}{2}}\big]P^{\frac{1}{2}}\\
=&P^{\frac{1}{2}}[I-\Lambda(I+\Lambda^{\top}\Lambda+\widetilde{\Lambda}^{\top}\widetilde{\Lambda})^{-1}\Lambda^{\top}]P^{\frac{1}{2}},
\end{aligned}
\end{equation*}
where
$\Lambda=P^{\frac{1}{2}}DR^{-\frac{1}{2}}$ and $\widetilde{\Lambda}=P^{\frac{1}{2}}D_0R^{-\frac{1}{2}}.$

From
$(I+\Lambda^{\top}\Lambda+\widetilde{\Lambda}^{\top}\widetilde{\Lambda})^{-1}\leq (I+\Lambda^{\top}\Lambda)^{-1}$,
we have
\begin{equation*}
\begin{aligned}
I-\Lambda(I+\Lambda^{\top}\Lambda+\widetilde{\Lambda}^{\top}\widetilde{\Lambda})^{-1}
\Lambda^{\top}\geq I-\Lambda(I+\Lambda^{\top}\Lambda)^{-1}\Lambda^{\top}.
\end{aligned}
\end{equation*}
By noting that $I-\Lambda(I+\Lambda^{\top}\Lambda)^{-1}\Lambda^{\top}=(I+\Lambda\Lambda^{\top})^{-1}$,
we have
\begin{equation*}
\begin{aligned}
P-PD\Sigma^{-1}D^{\top}P
\geq & P^{\frac{1}{2}}(I+\Lambda\Lambda^{\top})^{-1}P^{\frac{1}{2}}\\
=&P^{\frac{1}{2}}(I+P^{\frac{1}{2}}DR^{-1}D^{\top}P^{\frac{1}{2}})^{-1}P^{\frac{1}{2}}\geq 0.
\end{aligned}
\end{equation*}
Therefore, the following inequality holds
\begin{equation*}
C^{\top}(P-PD\Sigma^{-1}D^{\top}P)C\geq 0.
\end{equation*}
Similarly, it holds also that
\begin{equation*}
C_0^{\top}(P-PD_0\Sigma^{-1}D_0^{\top}P)C_0\geq 0.
\end{equation*}

Recall \eqref{condition for Pi}, by the standard result (see Theorem 7.2 in \cite{YZ1999}), we have that equation \eqref{Pi} admits a unique solution $\Pi(\cdot)\in C([0,T];\mathcal{S}_+^n)$.
By recalling  $\Pi=P+\Gamma$ and Lemma \ref{Plemma}, we get the well-posedness of equation \eqref{Gamma-1}.
\end{proof}

\begin{remark} Condition \eqref{condition for Pi} can be easily satisfied, for example, $C_0=0$ or $D_0=0$ or
$C_0=-C, D_0=D$ or $C_0=C, D_0=-D$.
\end{remark}

\begin{remark}\label{remark for partial information structure}
Comparing with Huang and Wang \cite{HW2016}, our diffusion terms of \eqref{state} can depend both on the state variable and control variable. This makes the Riccati equations are more challenging to solve. Our results generalize the one of \cite{HW2016} in non-trivial manner.
Moreover, if let $C=D=D_0=\beta=\beta_0=0$, our results will be consistent with \cite{HW2016}.
\end{remark}

\section{Application in Network Security Model}\label{sec4}
With the development of information technology, there exist high-frequency network communications. Accompanying that technology is the network security problems, and the  readers can refer to \cite{ST2016,ZP2021} and the reference therein for the study of network security model by using MFGs theory.
In this section, we investigate that how the investments of the users affect network security by applying our theoretical results.

Consider a network with $N$ users $\mathcal{A}_i$, for $1\leq i\leq N$. We suppose that each user $\mathcal{A}_i$ is willing to improve internet security level at first and individual users only know their own information.
Let $x_i(\cdot)$ be the safety level of the network, which is characterized by the following SDE driven by two mutually independent standard Brownian motions $W_i$, and $W_0$,
\begin{equation*}
\left\{
\begin{aligned}
dx_{i}(t)=&[ax_{i}(t)+bu_{i}(t)+\delta x^{(N)}(t)+k]dt\\
&+[cx_{i}(t)+du_{i}(t)+\sigma ]dW_{i}(t)+[d_0u_{i}(t)+\sigma_0]dW_{0 }(t), \\
x_{i}(0)=&x.
\end{aligned}
\right.
\end{equation*}
Here, $W_i$, and $W_0$ can be interpreted as the individual noise and common noise of user $\mathcal{A}_i$. The control input
$u_i(\cdot)$  of user $\mathcal{A}_i$  represents the investment in security maintenance, such as purchasing antivirus software and updating computer hardware. $x^{(N)}(\cdot)=\frac{1}{N}\sum_{j=1}^{N}x_j(\cdot)$ is the average security level of all users which represents that the levels of security of all users are interacted, such as if a user's level of security is reduced, he may spread the virus to other users when he communicates with them online.

Each user is devoted to minimizing the following performance criteria
\begin{equation*}
\mathcal{J}_{i}(u_{i}(\cdot ),u_{-i}(\cdot ))=\frac{1}{2}\mathbb{E}\Big\{\int_{0}^{T}\big[
q\big(x_{i}(t)-x^{(N)}(t)\big)^2+ ru_{i}^2(t)\big]dt+gx_{i}^2(T) \Big\},
\end{equation*}
where $u_{-i}=(u_1,\ldots,u_{i-1},u_{i+1},\ldots,u_N)$, $1\leq i\leq N$ and $q $, $g$ are nonnegative constants and $r>0$.
Obviously, this network security model is a kind of LQ MFGs studied in section \ref{sec2}.

By applying Theorem \ref{theoorem Riccati},
we obtain that the decentralized strategy is given by, for $1\leq i\leq N$,
\begin{equation}\label{net u}
\bar{u}_i(t)=-\frac{(bP(t)+cdP(t))\hat{\bar{z}}_i(t)+b\Gamma(t) \mathbb{E}[m(t)]+b\Phi(t)+d\sigma P(t)+d_0\sigma_0P(t)}{r+d^2P(t)+d_0^2P(t)},
\end{equation}
where $P(\cdot),\Gamma(\cdot),\Phi(\cdot),\mathbb{E}[m(\cdot)]$ solve respectively (see \eqref{P-1}-\eqref{m00},
\begin{equation*}
\begin{aligned}
&\dot{P}(t)+(2a+c^2)P(t)+q-\frac{(b+cd)^2P^2(t)}{r+d^2P(t)+d_0^2P(t)}=0,\\
&P(T)=g,
\end{aligned}
\end{equation*}
and
\begin{equation*}
\begin{aligned}
&\dot{\Gamma}(t)+2\Big(a-\frac{b(b+cd)P(t)}{r+d^2P(t)+d_0^2P(t)}\Big)\Gamma(t)\\
&+(P(t)+\Gamma(t))\delta -\frac{b^2\Gamma(t)^2}{r+d^2P(t)+d_0^2P(t)}-q=0,\\
&  \Gamma(T)=0,
\end{aligned}
\end{equation*}
and
\begin{equation*}
\begin{aligned}
&\dot{\Phi}(t)+(P(t)+\Gamma(t))k\\
&+\left(a-\frac{b(b+cd)P(t)}{r+d^2P(t)+d_0^2P(t)}-\frac{b^2\Gamma(t)}{r+d^2P(t)+d_0^2P(t)}\right)\Phi(t)\\
&+\left(c-\frac{d(b+cd)P(t)}{r+d^2P(t)+d_0^2P(t)}-\frac{bd\Gamma(t)}{r+d^2P(t)+d_0^2P(t)}\right)P(t)\sigma\\
&-\left(\frac{d_0(b+cd)P(t)}{r+d^2P(t)+d_0^2P(t)}+\frac{bd_0\Gamma(t)}{r+d^2P(t)+d_0^2P(t)}\right)P(t)\sigma_0
=0,\\
&\Phi(T)=0,
\end{aligned}
\end{equation*}
and
\begin{equation*}
\begin{aligned}\label{net Em}
d\mathbb{E}[m(t)]=&\Bigg\{\left[a+\delta-\frac{b(bP(t)+cdP(t)+b\Gamma(t))}{r+d^2P(t)+d_0^2P(t)}\right]\mathbb{E}[m(t)]\\
&+k-\frac{b(b\Phi(t)+d\sigma P(t)+d_0\sigma_0P(t))}{r+d^2P(t)+d_0^2P(t)}
\Bigg\}dt,\\
\mathbb{E}[m(0)]=&x.
\end{aligned}
\end{equation*}
The optimal filtering $\hat{\bar{z}}_i(\cdot)$ is given by
\begin{equation*}
\left\{
\begin{aligned}
d\hat{\bar{z}}_i(t)=&\Bigg\{\left(a-\frac{b(b+cd)P(t)}{r+d^2P(t)+d_0^2P(t)}\right)\hat{\bar{z}}_i(t)+\left(\delta-\frac{b^2\Gamma(t)}{r+d^2P(t)+d_0^2P(t)}\right)\mathbb{E}[m(t)]\\
&-\frac{b(b\Phi(t)+d\sigma P(t)+d_0\sigma_0P(t))}{r+d^2P(t)+d_0^2P(t)}+k\Bigg\}dt\\
&+\Bigg\{\left(c-\frac{d(b+cd)P(t)}{r+d^2P(t)+d_0^2P(t)}\right)\hat{\bar{z}}_i(t)-\frac{db\Gamma(t)}{r+d^2P(t)+d_0^2P(t)}\mathbb{E}[m(t)]\\
&-\frac{d(b\Phi(t)+d\sigma P(t)+d_0\sigma_0P(t))}{r+d^2P(t)+d_0^2P(t)}+\sigma\Bigg\}dW_i(t),\\
\hat{\bar{z}}_i(0)=&x.
\end{aligned}
\right.
\end{equation*}

In order to better illustrate the application in network security, and solve explicitly the Nash equilibrium, we set $a=b=\delta=\sigma=\sigma_0=r=g=1$, $c=d=d_0=k=0$ and $q=3$.
From \eqref{net u}, the decentralized strategy becomes
\begin{equation*}
\bar{u}_i(t)=-\big(P(t)\hat{\bar{z}}_i(t)+\Gamma(t)\mathbb{E}[m(t)]+\Phi(t)\big),
\end{equation*}
where
\begin{equation}\label{net P-L-Phi}
\left\{
\begin{aligned}
&\dot{P}(t)+2P(t)+3-P^2(t)=0,\\
&\dot{\Gamma}(t)+2\big(1-P(t)\big)\Gamma(t)
+\big(P(t)+\Gamma(t)\big) -\Gamma(t)^2-3=0,\\
&\dot{\Phi}(t)+\big(1-P(t)-\Gamma(t)\big)\Phi(t)=0,\\
&P(T)=1,\ \Gamma(T)=0,\ \Phi(T)=0,
\end{aligned}
\right.
\end{equation}
and
\begin{equation*}
\left\{
\begin{aligned}\label{m}
dm(t)&=\Big[2m(t)-\big(P(t)+\Gamma(t)\big)\mathbb{E}[m(t)]-\Phi(t)\Big]dt,\\
m(0)&=x.
\end{aligned}
\right.
\end{equation*}
Then, we have $m(t)=\mathbb{E}[m(t)]$, thus
\begin{equation*}
\left\{
\begin{aligned}
dm(t)&=\Big[\big(2-P(t)-\Gamma(t)\big)m(t)-\Phi(t)\Big]dt,\\
m(0)&=x.
\end{aligned}
\right.
\end{equation*}
Moreover, the optimal filtering $\hat{\bar{z}}_i(\cdot)$ is given by
\begin{equation*}
\left\{
\begin{aligned}
d\hat{\bar{z}}_i(t)=&\Big[\big(1-P(t)\big)\hat{\bar{z}}_i(t)+\big(1-\Gamma(t)\big)\mathbb{E}[m(t)]
-\Phi(t)\Big]dt\\
&+dW_i(t),\\
\hat{\bar{z}}_i(0)=&x.
\end{aligned}
\right.
\end{equation*}
From \eqref{net P-L-Phi}, it is easy to get
\begin{equation*}
P(t)=\frac{3-e^{4(t-T)}}{1+e^{4(t-T)}},\quad
%\Gamma(t)=\frac{3}{1+2e^{3(t-T)}}-\frac{3-e^{4(t-T)}}{1+e^{4(t-T)}},\quad
\Phi(t)=0, \quad t\in [0,T].
%\int_{t}^{T}\frac{3}{1+2e^{3(s-T)}}e^{\int_{t}^{s}\big(1-\frac{3}{1+2e^{3(r-T)}}\big)dr}ds,
\end{equation*}
Moreover, let $\Pi(\cdot):=P(\cdot)+\Gamma(\cdot)$, then $\Pi(\cdot)$ solves the following Riccati equation
\begin{equation*}
\left\{
\begin{aligned}
&\dot{\Pi}(t)+3\Pi(t)-\Pi^2(t)=0,\quad  t\in [0,T]\\
&\Pi(T)=1,
\end{aligned}
\right.
\end{equation*}
which yields that $\Gamma(t)=\Pi(t)-P(t)=\frac{3}{1+2e^{3(t-T)}}-\frac{3-e^{4(t-T)}}{1+e^{4(t-T)}}$.

In this setting, the decentralized strategy is given by
\begin{equation*}
\begin{aligned}
\bar{u}_i(t)=&\frac{e^{4(t-T)}-3}{1+e^{4(t-T)}}\big(\hat{\bar{z}}_i(t)-m(t)\big)
-\frac{3}{1+2e^{3(t-T)}}m(t),
\end{aligned}
\end{equation*}
where
\begin{equation*}
\begin{aligned}
m(t)=&xe^{2t-\int_{0}^{t}\frac{3}{1+2e^{3(s-T)}}ds},\\
\hat{\bar{z}}_i(t)=&xe^{t-\int_{0}^{t}\frac{3-e^{4(s-T)}}{1+e^{4(s-T)}}ds}
+\int_{0}^{t}e^{\int_{s}^{t}(1-\frac{3-e^{4(u-T)}}{1+e^{4(u-T)}})du}dW_i(s)\\
&+x\int_{0}^{t}\big(1-\frac{3}{1+2e^{3(t-T)}}+\frac{3-e^{4(s-T)}}{1+e^{4(s-T)}}\big)
e^{2s-\int_{s}^{t}\frac{3-e^{4(u-T)}}{1+e^{4(u-T)}}du-\int_{0}^{s}\frac{3}{1+2e^{3(u-T)}}du}ds.
\end{aligned}
\end{equation*}

Fig 1 gives the numerical solutions of $P(\cdot)$, $\Gamma(\cdot)$ and $\Phi(\cdot)$. One can see that $P(\cdot)$ is gradually decreasing, $\Phi(\cdot)$ is always 0. $\Gamma(\cdot)$ changes slowly at the beginning, and there is an upward trend near the terminal time. Fig 2 gives the numerical solution of $m(\cdot)$ which shows that  $m(\cdot)$ decreases first and then increases. Fig 3 presents the optimal filtering of decentralized states of 50 users. By comparing Fig 3 and Fig 2, we can find the overall trend of the states of 50 users is consistent with the trend  of $m(\cdot)$, moreover Fig 3 shows  clearly  the fluctuation of the individual noise. Fig 4 draws the control strategies of 50 users, we can find that the investment in security
maintenance of each user is increasing, although they may have different fluctuations.

In general, it is difficult to find an explicit solution. To better illustrate our results, we give some numerical simulations below. Suppose that $N=50$, $T=1$. Set the initial data as $x=1$, $a=1.5$, $b=2.8$, $\delta=1$, $k=2$, $c=0.6$, $d=2.5$, $d_0=6$, $\sigma=0.8$, $\sigma_0=0.3$, $q=3.3$, $r=2.5$, $g=5$.
Figure \ref{fig1} gives the numerical solutions of $P(\cdot)$, $\Gamma(\cdot)$ and $\Phi(\cdot)$. Figure \ref{fig2} gives the numerical solutions of $m(\cdot)$ and $\mathbb{E}[m(\cdot)]$.
Figure \ref{fig3} and Figure \ref{fig4} show, respectively, the optimal filtering of decentralized states and control strategies of 50 users.

\begin{figure}[h]
\begin{minipage}[t]{0.5\linewidth}
\centering
\includegraphics[width=3.3in]{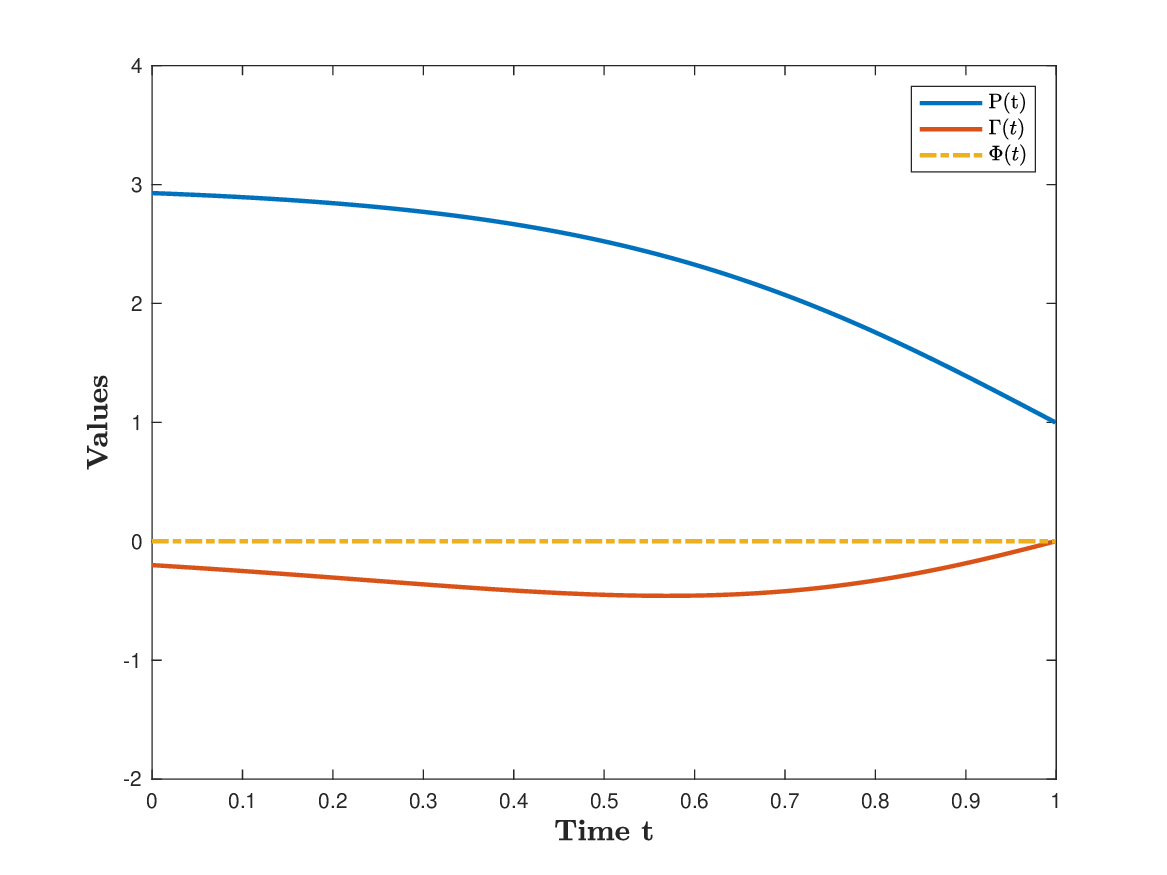}
\caption{The numerical solutions of Riccati equations $P(t)$, $\Gamma(t)$ and $\Phi(t)$.}
\label{fig1}
\end{minipage}%
\begin{minipage}[t]{0.5\linewidth}
\centering
\includegraphics[width=3.3in]{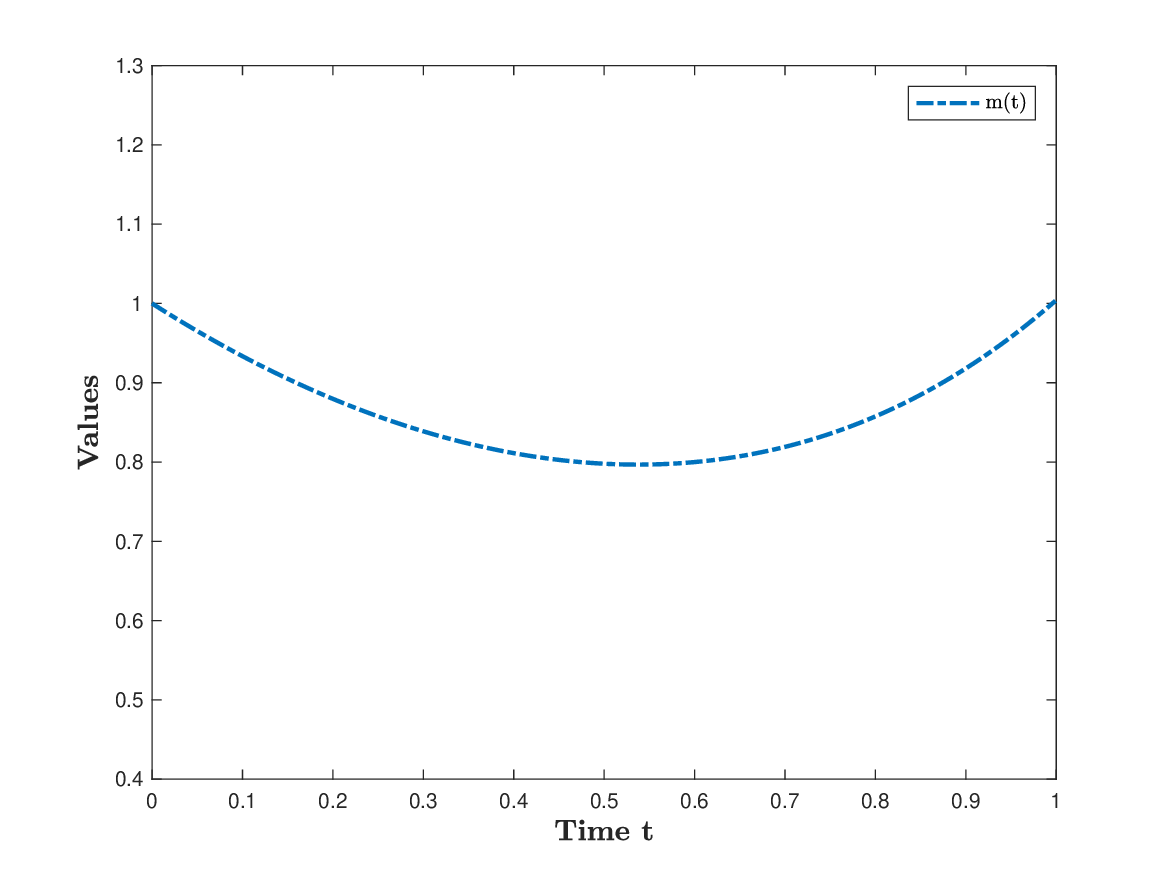}
\caption{The numerical solutions of state- average limit $m(t)$.}
\label{fig2}
\end{minipage}
\end{figure}

\begin{figure}[h]
\begin{minipage}[t]{0.5\linewidth}
\centering
\includegraphics[width=3.3in]{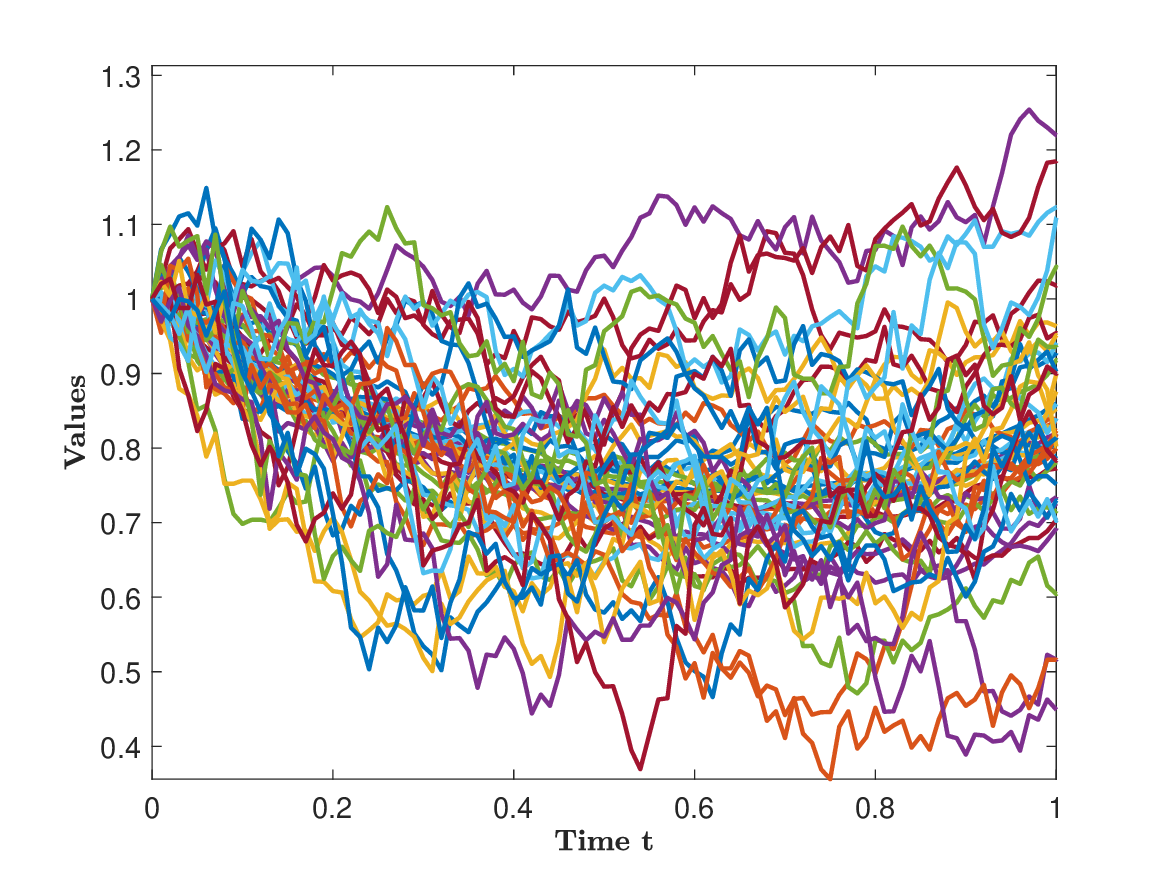}
\caption{The optimal filtering of decentralized states $\hat{\bar{z}}_i(t)$.}
\label{fig3}
\end{minipage}%
\begin{minipage}[t]{0.5\linewidth}
\centering
\includegraphics[width=3.3in]{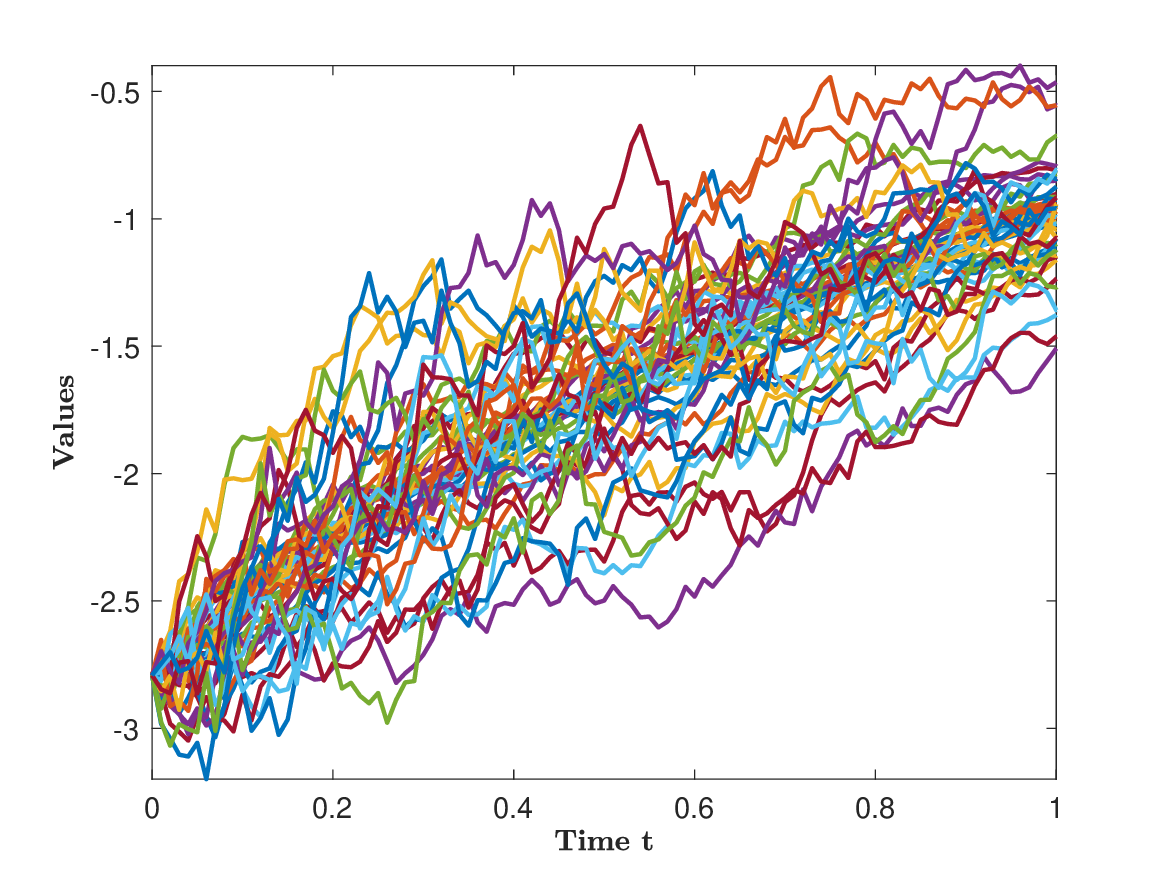}
\caption{The decentralized control strategies $\bar{u}_i(t)$.}
\label{fig4}
\end{minipage}
\end{figure}

\begin{figure}[h]
\begin{minipage}[t]{0.5\linewidth}
\centering
\includegraphics[width=3.3in]{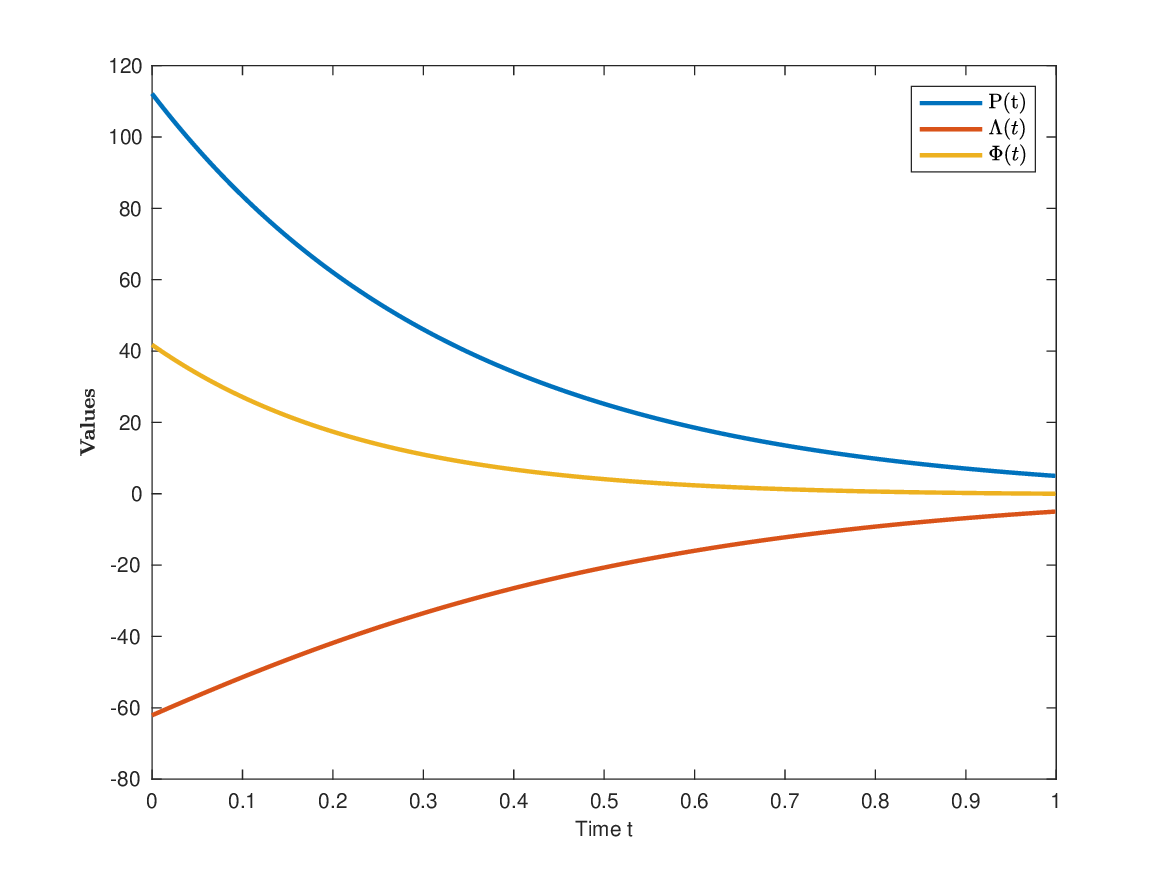}
\caption{The numerical solutions of Riccati equations $P(t)$, $\Gamma(t)$ and $\Phi(t)$.}
\label{fig1}
\end{minipage}%
\begin{minipage}[t]{0.5\linewidth}
\centering
\includegraphics[width=3.3in]{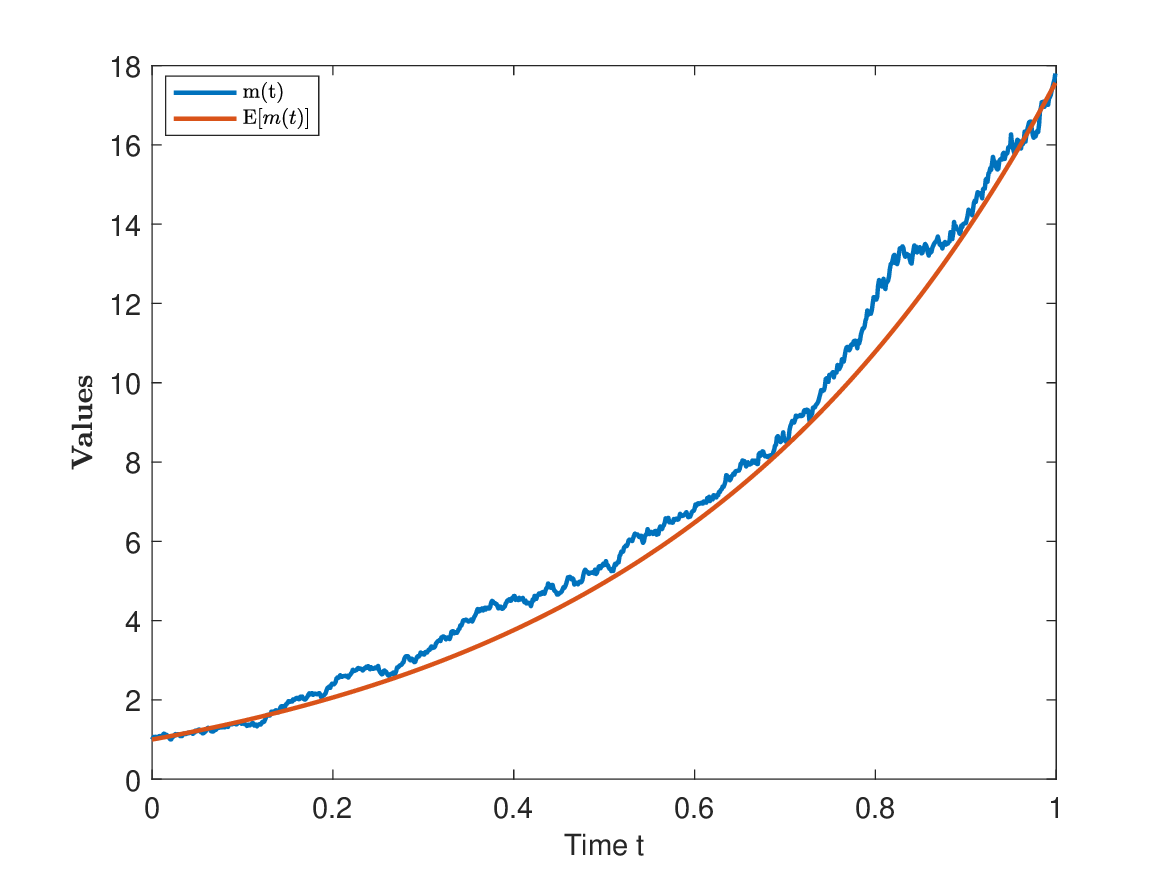}
\caption{The numerical solutions of state-average limit $m(t)$ and $\mathbb{E}[m(t)]$.}
\label{fig2}
\end{minipage}
\end{figure}

\begin{figure}[h]
\begin{minipage}[t]{0.5\linewidth}
\centering
\includegraphics[width=3.3in]{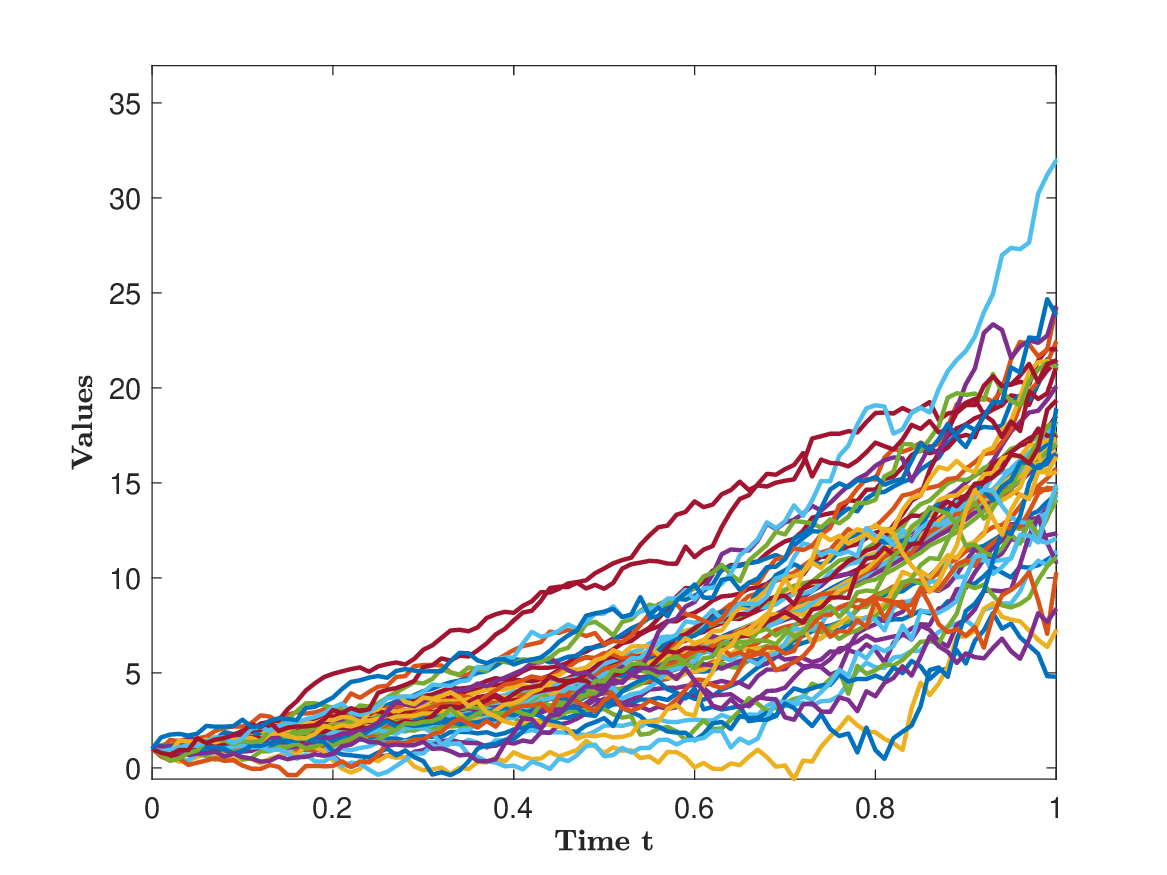}
\caption{The optimal filtering of decentralized states $\hat{\bar{z}}_i(t)$.}
\label{fig3}
\end{minipage}%
\begin{minipage}[t]{0.5\linewidth}
\centering
\includegraphics[width=3.3in]{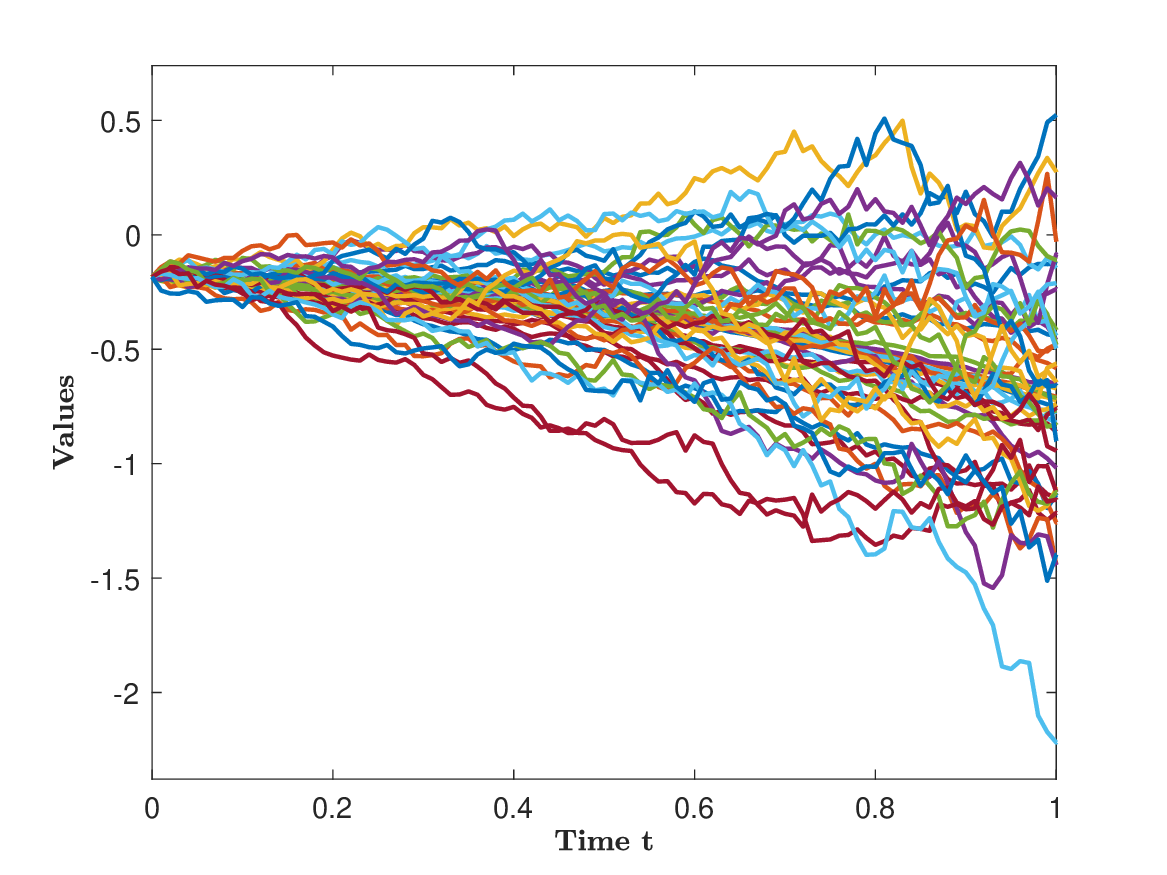}
\caption{The decentralized control strategies $\bar{u}_i(t)$.}
\label{fig4}
\end{minipage}
\end{figure}

\section*{Declaration of competing interest}
The authors declare that they have no known competing financial
interests or personal relationships that could have appeared
to influence the work reported in this paper.

\end{document}